\documentclass[preprint,3p,sort&compress,nopreprintline]{elsarticle}
\usepackage{geometry}
\usepackage[utf8]{inputenc}
\usepackage{amssymb}
\usepackage{amsmath}
\usepackage{amsthm}
\usepackage{tikz}

\graphicspath{{counterfigs/}}

\usetikzlibrary{calc}
\usepackage{graphicx}
\usepackage[colorinlistoftodos]{todonotes}
\usepackage[colorlinks=true, allcolors=blue]{hyperref}
\usepackage{tikz}
\usetikzlibrary{arrows}

\usepackage{xcolor} 
\usepackage{mathdots}
\usepackage{natbib}
\usepackage[normalem]{ulem}

\usepackage{subcaption}
\usepackage{multicol}
\usepackage[justification = centering]{caption}
\usepackage{float}
\usepackage[capitalise]{cleveref}

\newtheorem{theorem}{Theorem}     
\newtheorem{corollary}[theorem]{Corollary}
\newtheorem{lemma}[theorem]{Lemma}

\newtheorem{proposition}[theorem]{Proposition}
\newtheorem{conjecture}[theorem]{Conjecture}

\newtheorem{question}[theorem]{Question}

\newtheorem{claim}{Claim}
\newtheorem{example}{Example}

\newcommand{\im}[1]{\mu_{\rm int}(#1)}
\DeclareMathOperator{\Tr}{Tr}

\newcommand*{\myproofname}{Proof}
\newenvironment{myproof}[1][\myproofname]{\begin{proof}[#1]}{\end{proof}}

\begin{document}
    
\begin{frontmatter}

\title{On the Interval Coloring Impropriety of Graphs}
\author[1]{MacKenzie Carr\corref{cor1}} \ead{mackenzie\_carr@sfu.ca}
\author[2]{Eun-Kyung Cho} \ead{ekcho2020@gmail.com}
\author[3]{Nicholas Crawford} \ead{nicholas.2.crawford@ucdenver.edu}
\author[4a,4b]{Vesna Iršič}
\ead{vesna.irsic@fmf.uni-lj.si}
\author[5]{Leilani Pai} \ead{lpai@huskers.unl.edu}
\author[3]{Rebecca Robinson} \ead{rebecca.j.robinson@ucdenver.edu}
\date{\today}

\cortext[cor1]{Corresponding author}

\affiliation[1]{organization={Department of Mathematics, Simon Fraser University},
addressline={8888 University Drive},
city={Burnaby BC},
postcode={V5A 1S6},
country={Canada}
}
\affiliation[2]{organization={Department of Mathematics, Hankuk University of Foreign Studies},
city={Yongin-si, Gyeonggi-do},
country={Republic of Korea}
}
\affiliation[3]{organization={Department of Mathematics, University of Colorado Denver},
addressline={1201 Larimer Street},
city={Denver},
postcode={80204},
country={USA}
}
\affiliation[4a]{organization={Faculty of Mathematics and Physics, University of Ljubljana},
addressline={Jadranska ulica 19},
city={Ljubljana},
country={Slovenia}
}
\affiliation[4b]{organization={Institute of Mathematics, Physics and Mechanics, Ljubljana, Slovenia},
addressline={Jadranska ulica 19},
city={Ljubljana},
country={Slovenia}
}
\affiliation[5]{organization={Department of Mathematics, University of Nebraska-Lincoln},
addressline={Avery Hall 203},
city={Lincoln, NE},
postcode={68588},
country={USA}
}

\begin{abstract}
    An {\em improper interval (edge) coloring} of a graph $G$ is an assignment of colors to the edges of $G$ satisfying the condition that, for every vertex $v \in V(G)$, the set of colors assigned to the edges incident with $v$ forms an integral interval. 
    An interval coloring is {\em $k$-improper} if at most $k$ edges with the same color all share a common endpoint. 
    The minimum integer $k$ such that there exists a $k$-improper interval coloring of the graph $G$ is the {\em interval coloring impropriety} of $G$, denoted by $\mu_{int}(G)$.
    In this paper, we provide a construction of an interval coloring of a subclass of complete multipartite graphs. This provides additional evidence to the conjecture by Casselgren and Petrosyan that $\mu_{int}(G)\leq 2$ for all complete multipartite graphs $G$.
    Additionally, we determine improved upper bounds on the interval coloring impropriety of several classes of graphs, namely 2-trees, iterated triangulations, and outerplanar graphs. Finally, we investigate the interval coloring impropriety of the corona product of two graphs, $G\odot H$.    
\end{abstract}


\end{frontmatter}
\section{Introduction}

Let $G$ be a finite graph with vertex set $V(G)$, and edge set $E(G)$. An \emph{edge coloring} of a graph $G$ is a function $\phi: E(G) \to \mathcal{C}$ assigning labels, or colors, from a set $\mathcal{C}$, normally $\mathbb{N}$, to the edges of $G$. An edge coloring is \emph{proper} if the colors are assigned in a way so that adjacent edges (those sharing an endpoint) in $G$ receive different colors. Otherwise, the edge coloring is \emph{improper}. 

Asratian and Kamalian \cite{asratian1994investigation,asratian1987russian} introduced the problem of interval $k$-coloring a graph as a model for creating timetables without gaps for teachers and students. A proper edge coloring of a graph $G$ is said to be an \emph{interval} $k$\emph{-coloring} (also called a consecutive coloring, for example in \cite{giaro1997consecutive}) if exactly $k$ colors are used to color the edges of $G$ and the colors assigned to the edges incident to each vertex $v\in V(G)$ are distinct and form a set of consecutive integers. Note that there exist graphs that do not have an interval coloring, for example complete graphs $K_{n}$ for odd $n \geq 3$.
This leads to the question of whether a given graph has an interval coloring, which has been shown to be NP-complete to determine for bipartite graphs \cite{sevastjanov1990interval}. This problem has been investigated for several classes of graphs, including trees, regular graphs and complete bipartite graphs \cite{hansen1992scheduling, kamalian1989interval}, subclasses of bipartite graphs \cite{GIARO2004outerplanar,hansen1992scheduling,Petrosyan2014bipartite}, outerplanar graphs with maximum degree 3 \cite{Petrosyan2017outerplanar}, Cartesian, strong, lexicographic and tensor products and compositions of graphs \cite{giaro1997consecutive,Petrosyan2011products,Petrosyan2013cartesian,Tepanyan2017composition}, complete tripartite graphs $K_{1,m,n}$ \cite{Grzesik2014K1mn}, and biregular graphs \cite{Asratian2007biregular,Asratian2009biregular,Casselgren2011biregular,Casselgren2017biregular,Casselgren2015biregular,Hanson1998biregular,jing2020remarks,Pyatkin2004biregular,Yang2011biregular}. However, there remain many classes of graphs for which it is not known whether an interval coloring exists for all graphs in the class, including $(3,4)$-biregular graphs.

Hud{\'a}k et al.~\cite{hudak2016improper} introduced \emph{improper interval (edge) colorings}, motivated by scheduling problems without gaps in the schedule, but with some degree of conflict permitted. An improper edge coloring of a graph $G$ is said to be an \emph{improper interval (edge) coloring} if the colors assigned to the edges incident with each vertex $v\in V(G)$ form an integral interval. Subsequently, Casselgren and Petrosyan \cite{casselgren2021improper} introduced the \emph{interval coloring impropriety} of a graph. An improper interval coloring is $k$\emph{-improper} if at most $k$ edges with the same color all share a common endpoint. The smallest integer $k$ such that a graph $G$ has a $k$-improper interval coloring is called the \emph{interval coloring impropriety} (or simply \emph{impropriety}) of $G$, denoted by $\mu_{int}(G)$. Note that assigning every edge of $G$ the same color produces a $\Delta(G)$-improper interval coloring, so the interval coloring impropriety is well-defined for every graph. If $G$ has a proper interval coloring, then $\mu_{int}(G)=1$. 

In this paper, we investigate the impropriety of several classes of graphs. In Section 2, we begin with some preliminary definitions and results. 
In Section 3, we consider complete multipartite graphs, providing a construction of a coloring that shows that a particular subclass of complete multipartite graphs is interval colorable. In Section 4, we provide an upper bound on the impropriety of a 2-tree, and on the impropriety of the $n$-th iterated triangulation, both in terms of the maximum degree of the graph. In Section 5, we improve upon the upper bound on the impropriety of an outerplanar graph given in \cite{casselgren2021improper}. In Section 6, we determine upper bounds on the impropriety of the corona product of two graphs. Finally, in Section 7, we give some directions for further study of the improper interval impropriety of a graph and pose some open problems.

\section{Preliminaries}

For a positive integer $n$, we use $[n]$ to denote the set $\{1, 2, \dots, n\}$.
All graphs are assumed to be simple and finite, unless otherwise indicated. For a vertex $v$ in a graph $G$, we use $N_G(v)$ to denote the neighborhood of $v$ in $G$, and $d_G(v)$ to denote the degree of $v$ in $G$, or $N(v)$ and $d(v)$, respectively, when the graph is obvious from context. 
For a positive integer $k$, if $v \in V(G)$ satisfies $d_G(v) = k$, then we say that $v$ is a {\it $k$-vertex} in $G$.
The maximum degree of a vertex in $G$ is denoted by $\Delta(G)$ and the minimum degree of a vertex in $G$ by $\delta(G)$. A {\it clique} is a set of pairwise adjacent vertices in a graph. A complete graph, denoted by $K_n$, is a graph of order $n$ in which $V(K_n)$ forms a clique. A {\it path} in a graph $G$ is denoted by $v_1 v_2 \dots v_n$, where $v_i v_{i+1} \in E(G)$ for $i \in [n-1]$. We denote the path graph with $n$ vertices by $P_n$. Similarly, a {\it cycle} in a graph $G$ is denoted by $v_1 v_2 \dots v_n v_1$, where $v_i v_{i+1} \in E(G)$ for $i \in [n-1]$ and $v_n v_1 \in E(G)$.  
The cycle graph with $n$ vertices is denoted by $C_n$. 
For a positive integer $k \ge 3$, we refer to $C_k$ as a {\it $k$-cycle}.
A {\it star} $S_n$ is a graph that consists of a center vertex $v$ and leaves $u_1, u_2, \dots, u_n$ with $v u_i \in E(S_n)$ for $i \in [n]$. Note that $S_n = K_{1,n}$. A {\it spider graph} is a tree with a center vertex of degree at least 3, and all other vertices with degree at most 2. 
The family of spider graphs whose center vertex has degree exactly $k$ is denoted by $SP_k$. A {\it caterpillar graph} is a tree for which the removal of all leaves produces a path. A {\it wheel graph} is a graph formed by connecting a single universal vertex to all vertices of a cycle. 
For a positive integer $n$, we use $W_n$ to denote a wheel graph with $n$ vertices.

A {\it planar graph} is a graph which can be drawn in the plane without any edges crossing. A {\it plane graph} is a planar graph with a specific planar embedding. 
Thus, a plane graph divides the plane into a set of regions, called {\it faces}.
Each face is bounded by a closed walk called the {\it boundary of the face}.
Each plane graph has exactly one unbounded face, and we call such a face the {\it outer face}.
All the other bounded faces of a plane graph are called {\it inner faces}.
An {\it outerplanar graph} is a planar graph that has an embedding in which all vertices are on the boundary of the outer face. For all graph theory terminology not defined here, refer to \cite{bondy2010graph}.

The \emph{chromatic index} of a graph $G$, denoted by $\chi^\prime(G)$, is the minimum number of colors required to construct a proper edge coloring of $G$. Next, we state some existing results about the chromatic index and its relationship with impropriety.

\begin{theorem}
[Vizing’s Theorem] For any graph $G$, either $\chi^\prime(G) = \Delta(G)$ or $\chi^\prime(G) = \Delta(G)+1$. 
\end{theorem}

Graphs with $\chi^\prime(G) = \Delta(G)$ are called \emph{Class 1} and those with $\chi^\prime(G) = \Delta(G)+1$ are called \emph{Class 2}. There are some graph classes that are known to be Class 1. 

\begin{theorem}
[K{\"o}nig’s Edge Coloring Theorem] If a graph $G$ is bipartite, then $G$ is Class 1. 
\end{theorem}

Knowing whether a graph is Class 1 or Class 2 can help with determining whether the graph has an interval coloring or, if not, what the impropriety of the graph is. 

\begin{theorem}[\cite{asratian1994investigation}]
If a graph $G$ is interval colorable (i.e. $\mu_{int}(G) =1$), then $G$ is Class 1. 
\end{theorem}

Equivalently, a graph that is Class 2 does not admit an interval coloring, and thus has impropriety at least two. 

\begin{theorem}[\cite{casselgren2021improper}]
If $G$ is a regular graph of Class $i$, then $\mu_{int}(G) = i$ for $i \in \{1, 2\}$. 
\end{theorem}

As a corollary, this result implies that complete graphs and cycles of even order are interval colorable, while those of odd order have impropriety 2. 

For general graphs, it is known that the impropriety is at most two if the degree of every vertex is small. 

\begin{theorem}[\cite{casselgren2021improper}]
    \label{thm:delta5}
    If $G$ is a graph with $\Delta(G) \leq 5$, then $\mu_{int}(G) \leq 2$. 
\end{theorem}

Finally, we make use of the following results that determine the impropriety of some well-known classes of graphs, whose maximum degree is unbounded.

\begin{theorem} [\cite{kamalian1989interval}]
\label{thm:forest} 
    If $T$ is a forest, then $T$ is interval colorable (i.e.~$\mu_{int}(T)=1$).
\end{theorem}

\begin{theorem}[\cite{casselgren2021improper}]
\label{thm:wheels}
For a wheel graph $W_n$, $n \geq 4$, 
$\mu_{int}(W_n) = 1$ if and only if $n\in\{4,7,10\}$. Otherwise, $\mu_{int}(W_n) = 2$. 
\end{theorem}

\section{A subclass of complete multipartite graphs}

The impropriety of complete multipartite graphs was previously studied in \cite{casselgren2021improper}. We recall the following results.

\begin{theorem}[{\cite{casselgren2021improper}}]
    \label{thm:multipartite-r/2}
    If $n_1, \ldots, n_r \geq 1$, then $\im{K_{n_1, \ldots, n_r}} \leq \left \lceil \frac{r}{2} \right \rceil$.
\end{theorem}

\begin{conjecture}[{\cite{casselgren2021improper}}]
    \label{conj:multipartite-2}
    If $G$ is a complete $r$-partite graph, then $\im{G} \leq 2$.
\end{conjecture}

Note that, if true, Conjecture \ref{conj:multipartite-2} is best possible since there exist complete multipartite graphs of Class 2, thus with impropriety at least 2. In this section, we provide further support for the conjecture by studying complete multipartite graphs $K_{s,t,s,t,\ldots,s,t}$. Note that, by Theorem~\ref{thm:multipartite-r/2}, complete multipartite graphs with at most four parts satisfy Conjecture~\ref{conj:multipartite-2}. The main value of Theorem~\ref{thm:multipartite-special} described below is that it provides a family of complete multipartite graphs with an arbitrarily large number of parts that still satisfy Conjecture~\ref{conj:multipartite-2}.

 \emph{Parts} of a complete multipartite graph are its maximal independent sets. Let $G$ be a graph $K_{s,t,s,t,\ldots,s,t}$ with $k$ parts of size $s$ and $k$ parts of size $t$. 
 We say that $G$ has $m=2k$ parts.
 Then, we denote the parts of $G$ by $A_1, B_1, \ldots, A_k, B_k$, where $|A_i| = s$ and $|B_i| = t$ for every $i \in [k]$. The set of edges between parts $A$ and $B$ 
 is denoted by $E(A B)$ 
 to simplify the expressions in the rest of the section.
If $A$ and $B$ are parts of $G$ where $V(A) = \{x_1, \ldots, x_s\}$ and $V(B)=\{y_1,\ldots,y_t\}$, then we say that we color $E(AB)$ \emph{sequentially with shift $\alpha$} if the edge $x_i y_j$ is assigned color $\alpha + i + j -1$ for every $i \in [s]$, $j \in [t]$. Notice that, in this case, colors on edges incident to $x_i$ are $\alpha + i, \ldots, \alpha + i+t-1$, and colors on edges incident to $y_j$ are $\alpha + j, \ldots, \alpha + j+s-1$. In matrix form, the sequential coloring with shift $\alpha$ between parts $A$ and $B$ is presented in Table \ref{tab:sequential}.

  Suppose that vertices of $A$ and $B$ are incident with some already colored edges. We will say for short that we \emph{extend the coloring sequentially} to color $E(AB)$ in the following way. Let $\beta + i - 1$ be the maximum color incident with $x_i$ ($i \in [s]$), and let $\beta + j - 1$ be the maximum color incident with $y_j$ ($j \in [t]$). Then color $E(AB)$ sequentially with shift $\beta$.  

\begin{table}[h]
    \centering
    \scalebox{0.8}{
    \begin{tabular}{c||c c c c |}
         & $x_1$ & $x_2$ & $\cdots$ & $x_s$ \\ \hline\hline
        $y_1$ & $\alpha + 1$ & $\alpha + 2$ & $\cdots$ & $\alpha + s$ \\
        $y_2$ & $\alpha + 2$ & $\alpha + 3$ & $\cdots$ & $\alpha + s + 1$ \\
        $\vdots$ & $\vdots$ &  &  & $\vdots$ \\
        $y_t$ & $\alpha + t$ & $\alpha + t + 1$ & $\cdots$ & $\alpha + s + t - 1$ \\ \hline
    \end{tabular}}
    \caption{The sequential coloring with shift $\alpha$ in matrix form.}
    \label{tab:sequential}
\end{table}

Additionally, in the following proofs, we will assume that the vertices are labelled (since this labelling is needed for the sequential coloring to be well-defined). More precisely, if the parts of the graph $K_{s,t,s,t,\ldots,s,t}$ are $A_1$, $B_1$, $\ldots$, $A_k$, $B_k$, where $|A_i| = s$, $|B_i|=t$, then let $V(A_i) = \{x_{i,1}, \ldots, x_{i,s}\}$ and $V(B_i) = \{y_{i,1}, \ldots, y_{i,t}\}$ for all $i \in [k]$.

It follows from Theorem~\ref{thm:multipartite-r/2} that $\im{K_{s,t}}=1$ for $s, t \geq 1$. To present the basic case of the coloring used later in the section, we recall the proof of this result. Let $(A, B)$ be the bipartition of $G$ and vertices labelled as described above. Color $E(AB)$ sequentially with shift 0, and observe that this is a proper interval coloring of $G$. Now, we extend this approach to complete multipartite graphs with a larger number of parts.

\begin{lemma}
\label{lem:Kstst}
    For $s,t\geq 1$, the graph $K_{s,t,s,t}$ is interval colorable.
\end{lemma}

\begin{proof}
    Using the notation defined above, for $i \in [2]$, color $E(A_i B_i)$ sequentially with shift 0, color $E(A_1A_2)$ sequentially with shift $t$, 
    color $E(A_1 B_2)$ sequentially with shift $s+t$, color $E(B_1B_2)$ sequentially with shift $s$, and color $E(A_2B_1)$ sequentially with shift $s+t$.
    For an example with $s=4$ and $t=3$, see Table \ref{tab:K_4343}.

\begin{table}[h]
    \centering
    \scalebox{0.8}{
    \begin{tabular}{c||c c c c|c c c|c c c c|c c c|}
         & $x_{1,1}$ & $x_{1,2}$ & $x_{1,3}$ & $x_{1,4}$ & $y_{1,1}$ & $y_{1,2}$ & $y_{1,3}$ & $x_{2,1}$ & $x_{2,2}$ & $x_{2,3}$ & $x_{2,4}$ & $y_{2,1}$ & $y_{2,2}$ & $y_{2,3}$ \\ \hline\hline
        $x_{1,1}$ & & & & & 1 & 2 & 3 & 4 & 5 & 6 & 7 & 8 & 9 & 10 \\  
        $x_{1,2}$ & & & & & 2 & 3 & 4 & 5 & 6 & 7 & 8 & 9 & 10 & 11 \\
        $x_{1,3}$ & & & & & 3 & 4 & 5 & 6 & 7 & 8 & 9 & 10 & 11 & 12 \\
        $x_{1,4}$ & & & & & 4 & 5 & 6 & 7 & 8 & 9 & 10 & 11 & 12 & 13 \\ \hline
        $y_{1,1}$ & 1 & 2 & 3 & 4 & & & & 8 & 9 & 10  & 11 & 5 & 6 & 7 \\
        $y_{1,2}$ & 2 & 3 & 4 & 5 & & & & 9 & 10 & 11  & 12 & 6 & 7 & 8 \\
        $y_{1,3}$ & 3 & 4 & 5 & 6 & & & & 10 & 11 & 12  & 13 & 7 & 8 & 9 \\ \hline
        $x_{2,1}$ & 4 & 5 & 6 & 7 & 8 & 9 & 10 & & & & & 1 & 2 & 3 \\
        $x_{2,2}$ & 5 & 6 & 7 & 8 & 9 & 10 & 11 & & & & & 2 & 3 & 4 \\
        $x_{2,3}$ & 6 & 7 & 8 & 9 & 10 & 11 & 12 & & & & & 3 & 4 & 5 \\
        $x_{2,4}$ & 7 & 8 & 9 & 10 & 11 & 12 & 13 & & & & & 4 & 5 & 6 \\ \hline
        $y_{2,1}$ & 8 & 9 & 10 & 11 & 5 & 6 & 7 & 1 & 2 & 3 & 4 & & & \\
        $y_{2,2}$ & 9 & 10 & 11 & 12 & 6 & 7 & 8 & 2 & 3 & 4 & 5 & & & \\
        $y_{2,3}$ & 10 & 11 & 12 & 13 & 7 & 8 & 9 & 3 & 4 & 5 & 6 & & & \\ \hline
    \end{tabular}}
    \caption{The edge-coloring of $K_{4,3,4,3}$.}
    \label{tab:K_4343}
\end{table}

    Using the above coloring, the colors on edges incident with $x_{i,1}$ are $i, \ldots, i+t-1$, $i+t, \ldots, i+s+t-1$, $i+s+t, \ldots, i+s+2t-1$, which is an interval with no color occurring more than once. The interval property can be similarly verified at each of the other vertices. Thus, $\im{K_{s,t,s,t}}=1$. 
\end{proof}

\begin{theorem}
\label{thm:multipartite-special}
    If $m=2^{\ell}$ for some $\ell\geq 1$, then the complete multipartite graph $K_{s,t,s,t,\dots,s,t}$ with $m$ parts is interval colorable. 
\end{theorem}

\begin{proof}

For the cases where $m=2$ or $m=4$, see Lemma   \ref{lem:Kstst} and the example before it. We give an algorithmic proof for $m= 2^{\ell+1}$ when $\ell \ge 2$ 
that the graph $K_{s,t,\ldots, s,t}$ with $m$ parts is interval colorable. 
First, we describe some general conventions used in the rest of the proof.

For the sake of notation, we use matrices to describe the set of colors assigned to edges between parts. We will define an $n\times n$ matrix (with $n$ being the number of vertices in the graph) with rows and columns indexed by the vertices in the graph. We will call this matrix the coloring matrix. Note that, by symmetry, we need only describe the entries of the coloring matrix above the main diagonal, since the entries below the main diagonal are obtained by transposing the coloring matrix.
For $i\in [\frac{m}{2}]$, let $A_i$ and $B_i$ be the parts of size $s$ and $t$, respectively. The rows and columns of the coloring matrix will follow the ordering $A_1,B_1,A_2,B_2,\dots,A_{\frac{m}{2}},B_{\frac{m}{2}}$. Denote by $0_s$ and $0_t$ the zero matrices of size $s\times s$ and $t\times t$, respectively. Let $X_{s,k}$ be the $s\times s$ matrix recording the sequential coloring with shift $(k-1)s+kt$, for $k\geq 1$, on the edges between two parts of size $s$. 
Let $X_{t,k}$ be the $t\times t$ matrix recording the sequential coloring with shift $ks+(k-1)t$, for $k\geq 1$, on the edges between two parts of size $t$. 
Let $Y_{s,k}$ (similarly, $Y_{t,k}$) be the $s\times t$ (resp.~$t\times s$) matrix recording the sequential coloring with shift $ks+kt$, for $k\geq 0$, on the edges between a part of size $s$ and a part of size $t$. For example, see Table~\ref{tab:K_stst}. The rows and columns of each of these matrices are indexed by the vertices in the corresponding parts. Note that a particular row or column of the coloring matrix gives the colors assigned to all of the edges incident with the corresponding vertex. 

\begin{table}[h]
    \centering
    \scalebox{0.8}{
    \begin{tabular}{c||c c c c|c c c|c c c c|c c c|}
         & $A_1$ & $B_1$ & $A_2$ & $B_2$ \\ \hline\hline
        $A_1$ & $0_s$ & $Y_{s,0}$ & $X_{s,1}$ & $Y_{s,1}$ \\  
        $B_1$ & $Y_{t,0}$ & $0_t$ & $Y_{t,1}$ & $X_{t,1}$ \\
        $A_2$ & $X_{s,1}$ & $Y_{s,1}$ & $0_s$ & $Y_{s,0}$ \\
        $B_2$ & $Y_{t,1}$ & $X_{t,1}$ & $Y_{t,0}$ & $0_t$ \\ \hline
    \end{tabular}}
    \caption{The edge-coloring of $K_{s,t,s,t}$. }
    \label{tab:K_stst}
\end{table}

Now, in a particular row of $Y_{s,k}$, the colors that occur in that row are $ks+kt+i,ks+kt+i+1,\dots,ks+kt+i+t-1=ks+(k+1)t+i-1$, as in the definition of sequential coloring. Similarly, those that occur in a particular row of $X_{s,k+1}$ are $ks+(k+1)t+i,ks+(k+1)t+i+1,\dots,ks+(k+1)t+i+s-1=(k+1)s+(k+1)t+i-1$. From this, it is clear that the colors in any row of the coloring matrix formed from submatrices $Y_{s,0},X_{s,1},Y_{s,1},\dots,X_{s,r},Y_{s,r}$, for some $r\geq 1$, form a set of consecutive integers. 

Additionally, in a particular row of $Y_{t,k}$, the colors that occur in that row are $ks+kt+j,ks+kt+j+1,\dots,ks+kt+j+s-1=(k+1)s+kt+j-1$, as in the definition of sequential coloring. Similarly, those that occur in a particular row of $X_{t,k+1}$ are $(k+1)s+kt+j,(k+1)s+kt+j+1,\dots,(k+1)s+kt+j+t-1=(k+1)s+(k+1)t+j-1$. From this, it is clear that the colors in any row of the coloring matrix formed from submatrices $Y_{t,0},X_{t,1},Y_{t,1},\dots,X_{t,r},Y_{t,r}$, for some $r\geq 1$, form a set of consecutive integers. 

A similar argument shows that any column formed by submatrices $Y_{s,0},X_{t,1},Y_{s,1},\dots,X_{t,r},Y_{s,r}$ and by $Y_{t,0},X_{s,1},Y_{t,1},\dots,X_{s,r},Y_{t,r}$, for some $r\geq 1$, also forms a set of consecutive integers. 

Therefore, we can use the matrices $X_{s,k},X_{t,k},Y_{s,k}$ and $Y_{t,k}$ to generate an interval coloring of the complete multipartite graph with the described number of parts as follows. 

Let $m=2^{\ell+1}$ for $\ell \ge 2$.
We partition the parts into two sub-parts $M_1 = \{A_1, B_1, \ldots, A_{\frac{m}{4}}, B_{\frac{m}{4}}\}$ and $M_2 = \{A_{\frac{m}{4}+1}, B_{\frac{m}{4}+1}, \ldots, A_{\frac{m}{2}}, B_{\frac{m}{2}}\}$.
By the algorithm described above, we color the edges with both ends in $M_i$, $i\in[2]$. 
Thus, all that remains is to color the edges with one endpoint in $M_1$ and the other in $M_2$.
To do this, we extend the coloring in the rows corresponding to $A_1$ and $B_1$ to $M_2$ sequentially using $X_{s,k}, X_{t,k}, Y_{s,k}$, and $Y_{t,k}$ with $k \in \{\frac{m}{4}, \ldots, \frac{m}{2}-1\}$. See Table~\ref{tab:m=8} for the case when $\ell=2$.

\begin{table}[h]
    \centering
    \scalebox{0.8}{
    \begin{tabular}{c||c c c c|c c c c|c c c c|c c c|}
         & $A_1$ & $B_1$ & $A_2$ & $B_2$ & $A_3$ & $B_3$ & $A_4$ & $B_4$\\ \hline\hline
        $A_1$ & $0_s$ & $Y_{s,0}$ & $X_{s,1}$ & $Y_{s,1}$ & $X_{s,2}$ & $Y_{s,2}$ & $X_{s,3}$ & $Y_{s,3}$\\  
        $B_1$ & $Y_{t,0}$ & $0_t$ & $Y_{t,1}$ & $X_{t,1}$ & $Y_{t,2}$ & $X_{t,2}$ & $Y_{t,3}$ & $X_{t,3}$\\
        $A_2$ & $X_{s,1}$ & $Y_{s,1}$ & $0_s$ & $Y_{s,0}$ & & & &  \\
        $B_2$ & $Y_{t,1}$ & $X_{t,1}$ & $Y_{t,0}$ & $0_t$ & & & &  \\ \hline
        $A_3$ & $X_{s,2}$ & $Y_{s,2}$ & & & $0_s$ & $Y_{s,0}$ & $X_{s,1}$ & $Y_{s,1}$ \\
        $B_3$ & $Y_{t,2}$ & $X_{t,2}$ & & & $Y_{t,0}$ & $0_t$ & $Y_{t,1}$ & $X_{t,1}$ \\
        $A_4$ & $X_{s,3}$ & $Y_{s,3}$ & & & $X_{s,1}$ & $Y_{s,1}$ & $0_s$ & $Y_{s,0}$ \\
        $B_4$ & $Y_{t,3}$ & $X_{t,3}$ & & & $Y_{t,1}$ & $X_{t,1}$ & $Y_{s,0}$ &$0_t$ \\ \hline
    \end{tabular}}
    \caption{Initial coloring step of $m=8$.}
    \label{tab:m=8}
\end{table}

Take each $2 \times 2$ submatrix with rows corresponding to $A_1$ and $B_1$ and copy it to color the diagonal and off-diagonal $2 \times 2$ blocks in the unfilled submatrix corresponding to the rows corresponding to $A_2$ and $B_2$. See Table~\ref{tab:K_stststst} for the case $\ell=2$ and Table~\ref{tab:m=16} for the case $\ell=3$. 

\begin{table}[h]
    \centering
    \scalebox{0.8}{
    \begin{tabular}{c||c c c c|c c c c|c c c c|c c c|}
        & $A_1$ & $B_1$ & $A_2$ & $B_2$ & $A_3$ & $B_3$ & $A_4$ & $B_4$\\ \hline\hline
        $A_1$ & $0_s$ & $Y_{s,0}$ & $X_{s,1}$ & $Y_{s,1}$ & $X_{s,2}$ & $Y_{s,2}$ & $X_{s,3}$ & $Y_{s,3}$\\  
        $B_1$ & $Y_{t,0}$ & $0_t$ & $Y_{t,1}$ & $X_{t,1}$ & $Y_{t,2}$ & $X_{t,2}$ & $Y_{t,3}$ & $X_{t,3}$\\
        $A_2$ & $X_{s,1}$ & $Y_{s,1}$ & $0_s$ & $Y_{s,0}$ & $X_{s,3}$ & $Y_{s,3}$ & $X_{s,2}$ & $Y_{s,2}$  \\
        $B_2$ & $Y_{t,1}$ & $X_{t,1}$ & $Y_{t,0}$ & $0_t$ & $Y_{t,3}$ & $X_{t,3}$ & $Y_{t,2}$ & $X_{t,2}$  \\ \hline
        $A_3$ & $X_{s,2}$ & $Y_{s,2}$ & $X_{s,3}$ & $Y_{s,3}$ & $0_s$ & $Y_{s,0}$ & $X_{s,1}$ & $Y_{s,1}$ \\
        $B_3$ & $Y_{t,2}$ & $X_{t,2}$ & $Y_{t,3}$ & $X_{t,3}$ & $Y_{t,0}$ & $0_t$ & $Y_{t,1}$ & $X_{t,1}$ \\
        $A_4$ & $X_{s,3}$ & $Y_{s,3}$ & $X_{s,2}$ & $Y_{s,2}$ & $X_{s,1}$ & $Y_{s,1}$ & $0_s$ & $Y_{s,0}$ \\
        $B_4$ & $Y_{t,3}$ & $X_{t,3}$ & $Y_{t,2}$ & $X_{t,2}$ & $Y_{t,1}$ & $X_{t,1}$ & $Y_{s,0}$ &$0_t$ \\ \hline
       
    \end{tabular}}
    \caption{Matrix showing the edge colors of $K_{s,t,s,t,s,t,s,t}$.}
    \label{tab:K_stststst}
\end{table}

\begin{table}[H]
    \centering
    \scalebox{0.8}{
    \begin{tabular}{c||c c c c|c c c c|c c c c|c c c c|}
         & $A_1$ & $B_1$ & $A_2$ & $B_2$ & $A_3$ & $B_3$ & $A_4$ & $B_4$ 
         & $A_5$ & $B_5$ & $A_6$ & $B_6$ & $A_7$ & $B_7$ & $A_8$ & $B_8$ \\ \hline\hline
        $A_1$ & $0_s$ & $Y_{s,0}$ & $X_{s,1}$ & $Y_{s,1}$ & $X_{s,2}$ & $Y_{s,2}$ & $X_{s,3}$ & $Y_{s,3}$ & $X_{s,4}$ & $Y_{s,4}$ & $X_{s,5}$ & $Y_{s,5}$ & $X_{s,6}$ & $Y_{s,6}$ & $X_{s,7}$ & $Y_{s,7}$\\  
        $B_1$ & $Y_{t,0}$ & $0_t$ & $Y_{t,1}$ & $X_{t,1}$ & $Y_{t,2}$ & $X_{t,2}$ & $Y_{t,3}$ & $X_{t,3}$ & $Y_{t,4}$ & $X_{t,4}$ & $Y_{t,5}$ & $X_{t,5}$ & $Y_{t,6}$ & $X_{t,6}$ & $Y_{t,7}$ & $X_{t,7}$\\
        $A_2$ & $X_{s,1}$ & $Y_{s,1}$ & $0_s$ & $Y_{s,0}$ & $X_{s,3}$ & $Y_{s,3}$ & $X_{s,2}$ & $Y_{s,2}$ & $X_{s,5}$ & $Y_{s,5}$ & $X_{s,4}$ & $Y_{s,4}$ & $X_{s,7}$ & $Y_{s,7}$ & $X_{s,6}$ & $Y_{s,6}$\\
        $B_2$ & $Y_{t,1}$ & $X_{t,1}$ & $Y_{t,0}$ & $0_t$ & $Y_{t,3}$ & $X_{t,3}$ & $Y_{t,2}$ & $X_{t,2}$ & $Y_{t,5}$ & $X_{t,5}$ & $Y_{t,4}$ & $X_{t,4}$ & $Y_{t,7}$ & $X_{t,7}$ & $Y_{t,6}$ & $X_{t,6}$\\ \hline
        $A_3$ & $X_{s,2}$ & $Y_{s,2}$ & $X_{s,3}$ & $Y_{s,3}$ & $0_s$ & $Y_{s,0}$ & $X_{s,1}$ & $Y_{s,1}$ & & & & & & & &  \\
        $B_3$ & $Y_{t,2}$ & $X_{t,2}$ & $Y_{t,3}$ & $X_{t,3}$ & $Y_{t,0}$ & $0_t$ & $Y_{t,1}$ & $X_{t,1}$ & & & & & & & &  \\
        $A_4$ & $X_{s,3}$ & $Y_{s,3}$ & $X_{s,2}$ & $Y_{s,2}$ & $X_{s,1}$ & $Y_{s,1}$ & $0_s$ & $Y_{s,0}$ & & & & & & & &  \\
        $B_4$ & $Y_{t,3}$ & $X_{t,3}$ & $Y_{t,2}$ & $X_{t,2}$ & $Y_{t,1}$ & $X_{t,1}$ & $Y_{t,0}$ & $0_t$ & & & & & & & &  \\ \hline
        $A_5$ & $X_{s,4}$ & $Y_{s,4}$ & $X_{s,5}$ & $Y_{s,5}$ & & & & & $0_s$ & $Y_{s,0}$ & $X_{s,1}$ & $Y_{s,1}$ & $X_{s,2}$ & $Y_{s,2}$ & $X_{s,3}$ & $Y_{s,3}$\\  
        $B_5$ & $Y_{t,4}$ & $X_{t,4}$ & $Y_{t,5}$ & $X_{t,5}$ & & & & & $Y_{t,0}$ & $0_t$ & $Y_{t,1}$ & $X_{t,1}$ & $Y_{t,2}$ & $X_{t,2}$ & $Y_{t,3}$ & $X_{t,3}$\\
        $A_6$ & $X_{s,5}$ & $Y_{s,5}$ & $X_{s,4}$ & $Y_{s,4}$ & & & & & $X_{s,1}$ & $Y_{s,1}$ & $0_s$ & $Y_{s,0}$ & $X_{s,3}$ & $Y_{s,3}$ & $X_{s,2}$ & $Y_{s,2}$\\
        $B_6$ & $Y_{t,5}$ & $X_{t,5}$ & $Y_{t,4}$ & $X_{t,4}$ & & & & & $Y_{t,1}$ & $X_{t,1}$ & $Y_{t,0}$ & $0_t$ & $Y_{t,3}$ & $X_{t,3}$ & $Y_{t,2}$ & $X_{t,2}$\\ \hline
        $A_7$ & $X_{s,6}$ & $Y_{s,6}$ & $X_{s,7}$ & $Y_{s,7}$ & & & & & $X_{s,2}$ & $Y_{s,2}$ & $X_{s,3}$ & $Y_{s,3}$ & $0_s$ & $Y_{s,0}$ & $X_{s,1}$ & $Y_{s,1}$ \\
        $B_7$ & $Y_{t,6}$ & $X_{t,6}$ & $Y_{t,7}$ & $X_{t,7}$ & & & & & $Y_{t,2}$ & $X_{t,2}$ & $Y_{t,3}$ & $X_{t,3}$ & $Y_{t,0}$ & $0_t$ & $Y_{t,1}$ & $X_{t,1}$ \\
        $A_8$ & $X_{s,7}$ & $Y_{s,7}$ & $X_{s,6}$ & $Y_{s,6}$ & & & & & $X_{s,3}$ & $Y_{s,3}$ & $X_{s,2}$ & $Y_{s,2}$ & $X_{s,1}$ & $Y_{s,1}$ & $0_s$ & $Y_{s,0}$ \\
        $B_8$ & $Y_{t,7}$ & $X_{t,7}$ & $Y_{t,6}$ & $X_{t,6}$ & & & & & $Y_{t,3}$ & $X_{t,3}$ & $Y_{t,2}$ & $X_{t,2}$ & $Y_{t,1}$ & $X_{t,1}$ & $Y_{t,0}$ & $0_t$ \\  \hline
    \end{tabular}}
    \caption{Initial coloring for $m=16$.}
    \label{tab:m=16}
\end{table}

Repeat the same method for the $4 \times 4$ submatrices corresponding to the rows $A_1,B_1,A_2,B_2$ to complete the submatrices corresponding to the rows $A_3,B_3,A_4,B_4$. See Table~\ref{tab:K_big} for the case when $\ell=4$. We repeat this process for $2^j \times 2^j$ submatrices for $j \in [\frac{m}{8}]$ corresponding to $A_1, B_1, \ldots, A_{j-1}, B_{j-1}$ until we fill in the whole matrix. 

 \begin{table}[H]
    \centering
    \scalebox{0.8}{
    \begin{tabular}{c||c c c c|c c c c|c c c c|c c c c|}
         & $A_1$ & $B_1$ & $A_2$ & $B_2$ & $A_3$ & $B_3$ & $A_4$ & $B_4$ 
         & $A_5$ & $B_5$ & $A_6$ & $B_6$ & $A_7$ & $B_7$ & $A_8$ & $B_8$ \\ \hline\hline
        $A_1$ & $0_s$ & $Y_{s,0}$ & $X_{s,1}$ & $Y_{s,1}$ & $X_{s,2}$ & $Y_{s,2}$ & $X_{s,3}$ & $Y_{s,3}$ & $X_{s,4}$ & $Y_{s,4}$ & $X_{s,5}$ & $Y_{s,5}$ & $X_{s,6}$ & $Y_{s,6}$ & $X_{s,7}$ & $Y_{s,7}$\\  
        $B_1$ & $Y_{t,0}$ & $0_t$ & $Y_{t,1}$ & $X_{t,1}$ & $Y_{t,2}$ & $X_{t,2}$ & $Y_{t,3}$ & $X_{t,3}$ & $Y_{t,4}$ & $X_{t,4}$ & $Y_{t,5}$ & $X_{t,5}$ & $Y_{t,6}$ & $X_{t,6}$ & $Y_{t,7}$ & $X_{t,7}$\\
        $A_2$ & $X_{s,1}$ & $Y_{s,1}$ & $0_s$ & $Y_{s,0}$ & $X_{s,3}$ & $Y_{s,3}$ & $X_{s,2}$ & $Y_{s,2}$ & $X_{s,5}$ & $Y_{s,5}$ & $X_{s,4}$ & $Y_{s,4}$ & $X_{s,7}$ & $Y_{s,7}$ & $X_{s,6}$ & $Y_{s,6}$\\
        $B_2$ & $Y_{t,1}$ & $X_{t,1}$ & $Y_{t,0}$ & $0_t$ & $Y_{t,3}$ & $X_{t,3}$ & $Y_{t,2}$ & $X_{t,2}$ & $Y_{t,5}$ & $X_{t,5}$ & $Y_{t,4}$ & $X_{t,4}$ & $Y_{t,7}$ & $X_{t,7}$ & $Y_{t,6}$ & $X_{t,6}$\\ \hline
        $A_3$ & $X_{s,2}$ & $Y_{s,2}$ & $X_{s,3}$ & $Y_{s,3}$ & $0_s$ & $Y_{s,0}$ & $X_{s,1}$ & $Y_{s,1}$ & $X_{s,6}$ & $Y_{s,6}$ & $X_{s,7}$ & $Y_{s,7}$ & $X_{s,4}$ & $Y_{s,4}$ & $X_{s,5}$ & $Y_{s,5}$ \\
        $B_3$ & $Y_{t,2}$ & $X_{t,2}$ & $Y_{t,3}$ & $X_{t,3}$ & $Y_{t,0}$ & $0_t$ & $Y_{t,1}$ & $X_{t,1}$ & $Y_{t,6}$ & $X_{t,6}$ & $Y_{t,7}$ & $X_{t,7}$ & $Y_{t,4}$ & $X_{t,4}$ & $Y_{t,5}$ & $X_{t,5}$\\
        $A_4$ & $X_{s,3}$ & $Y_{s,3}$ & $X_{s,2}$ & $Y_{s,2}$ & $X_{s,1}$ & $Y_{s,1}$ & $0_s$ & $Y_{s,0}$ & $X_{s,7}$ & $Y_{s,7}$ & $X_{s,6}$ & $Y_{s,6}$ & $X_{s,5}$ & $Y_{s,5}$ & $X_{s,4}$ & $Y_{s,4}$\\
        $B_4$ & $Y_{t,3}$ & $X_{t,3}$ & $Y_{t,2}$ & $X_{t,2}$ & $Y_{t,1}$ & $X_{t,1}$ & $Y_{t,0}$ & $0_t$ & $Y_{t,7}$ & $X_{t,7}$ & $Y_{t,6}$ & $X_{t,6}$ & $Y_{t,5}$ & $X_{t,5}$ & $Y_{t,4}$ & $X_{t,4}$ \\ \hline
        $A_5$ & $X_{s,4}$ & $Y_{s,4}$ & $X_{s,5}$ & $Y_{s,5}$ & $X_{s,6}$ & $Y_{s,6}$ & $X_{s,7}$ & $Y_{s,7}$ & $0_s$ & $Y_{s,0}$ & $X_{s,1}$ & $Y_{s,1}$ & $X_{s,2}$ & $Y_{s,2}$ & $X_{s,3}$ & $Y_{s,3}$\\  
        $B_5$ & $Y_{t,4}$ & $X_{t,4}$ & $Y_{t,5}$ & $X_{t,5}$ & $Y_{t,6}$ & $X_{t,6}$ & $Y_{t,7}$ & $X_{t,7}$ & $Y_{t,0}$ & $0_t$ & $Y_{t,1}$ & $X_{t,1}$ & $Y_{t,2}$ & $X_{t,2}$ & $Y_{t,3}$ & $X_{t,3}$\\
        $A_6$ & $X_{s,5}$ & $Y_{s,5}$ & $X_{s,4}$ & $Y_{s,4}$ & $X_{s,7}$ & $Y_{s,7}$ & $X_{s,6}$ & $Y_{s,6}$ & $X_{s,1}$ & $Y_{s,1}$ & $0_s$ & $Y_{s,0}$ & $X_{s,3}$ & $Y_{s,3}$ & $X_{s,2}$ & $Y_{s,2}$\\
        $B_6$ & $Y_{t,5}$ & $X_{t,5}$ & $Y_{t,4}$ & $X_{t,4}$ & $Y_{t,7}$ & $X_{t,7}$ & $Y_{t,6}$ & $X_{t,6}$ & $Y_{t,1}$ & $X_{t,1}$ & $Y_{t,0}$ & $0_t$ & $Y_{t,3}$ & $X_{t,3}$ & $Y_{t,2}$ & $X_{t,2}$\\ \hline
        $A_7$ & $X_{s,6}$ & $Y_{s,6}$ & $X_{s,7}$ & $Y_{s,7}$ & $X_{s,4}$ & $Y_{s,4}$ & $X_{s,5}$ & $Y_{s,5}$ & $X_{s,2}$ & $Y_{s,2}$ & $X_{s,3}$ & $Y_{s,3}$ & $0_s$ & $Y_{s,0}$ & $X_{s,1}$ & $Y_{s,1}$ \\
        $B_7$ & $Y_{t,6}$ & $X_{t,6}$ & $Y_{t,7}$ & $X_{t,7}$ & $Y_{t,4}$ & $X_{t,4}$ & $Y_{t,5}$ & $X_{t,5}$ & $Y_{t,2}$ & $X_{t,2}$ & $Y_{t,3}$ & $X_{t,3}$ & $Y_{t,0}$ & $0_t$ & $Y_{t,1}$ & $X_{t,1}$ \\
        $A_8$ & $X_{s,7}$ & $Y_{s,7}$ & $X_{s,6}$ & $Y_{s,6}$ & $X_{s,5}$ & $Y_{s,5}$ & $X_{s,4}$ & $Y_{s,4}$ & $X_{s,3}$ & $Y_{s,3}$ & $X_{s,2}$ & $Y_{s,2}$ & $X_{s,1}$ & $Y_{s,1}$ & $0_s$ & $Y_{s,0}$ \\
        $B_8$ & $Y_{t,7}$ & $X_{t,7}$ & $Y_{t,6}$ & $X_{t,6}$ & $Y_{t,5}$ & $X_{t,5}$ & $Y_{t,4}$ & $X_{t,4}$ & $Y_{t,3}$ & $X_{t,3}$ & $Y_{t,2}$ & $X_{t,2}$ & $Y_{t,1}$ & $X_{t,1}$ & $Y_{t,0}$ & $0_t$ \\  \hline
    \end{tabular}}
    \caption{Matrix showing the edge colors of $K_{s,t,\ldots,s,t}$ with $2^4 = 16$ parts.}
    \label{tab:K_big}
\end{table}

Now, each row of the coloring matrix has one occurrence of each of the matrices in either the sequence $Y_{s,0},X_{s,1},Y_{s,1},\dots,X_{s,r},Y_{s,r}$ or the sequence $Y_{t,0},X_{t,1},Y_{t,1},\dots,X_{t,r},Y_{t,r}$ for $r = \frac{m}{2}-1$ and, as described above, this corresponds to an integral interval of colors at the edges incident to each vertex. Thus, we have constructed an interval coloring of $K_{s,t,s,t,\dots,s,t}$ with $m$ parts. This completes the proof.
\end{proof}

Observe that the family of complete multipartite graphs from Theorem~\ref{thm:multipartite-special} provides complete multipartite graphs with arbitrarily many parts that are interval colorable.

\section{\texorpdfstring{$k$}{k}-trees}

A \emph{$k$-tree} is a graph obtained from $K_{k+1}$ by repeatedly adding vertices in such a way that each added vertex $v$ has exactly $k$ neighbors $U$, where $U$ is a clique. For example, 1-trees are exactly trees, 2-trees are maximal series-parallel graphs, and include the maximal outerplanar graphs. Planar 3-trees are also known as Apollonian networks.

Every tree has impropriety $1$, by Theorem~\ref{thm:forest}.
In this section, we study the impropriety of $2$-trees and of a subfamily of planar $3$-trees.

\subsection{2-trees}

We begin by determining an upper bound on the impropriety of a general 2-tree, before restricting our focus to two particular subclasses of 2-trees.

\begin{theorem}
\label{thm:2-tree}
    If $G$ is a 2-tree with $\Delta(G) \geq 3$, then $\im{G} \leq \left \lceil \frac{\Delta(G)}{3} \right \rceil$.
\end{theorem}

\begin{proof}
Let $G$ be a minimal counterexample to the statement with respect to the number of vertices. 
So $G$ is a $2$-tree with $\Delta(G) \ge 3$ and $\im{G} > \left\lceil \frac{\Delta(G)}{3} \right\rceil$.
If $\Delta(G) = 3$, then $G$ is a diamond graph which satisfies $\im{G} = 1 = \left\lceil \frac{3}{3} \right\rceil$, so it is not a counterexample; see Figure~\ref{fig:diamond}.

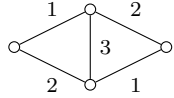
\begin{figure}[h!]
    \centering
    \begin{tikzpicture}
    \node[circle, draw=black, fill=white, inner sep =0.05cm](y) at (0,0){};
    \node[circle, draw=black, fill=white, inner sep =0.05cm](x) at (0,-1){};
    \node[circle, draw=black, fill=white, inner sep =0.05cm](v) at (1,-0.5){};
    \node[circle, draw=black, fill=white, inner sep =0.05cm](w) at (-1,-0.5){};
    \draw (x) -- (y);
    \draw (y) -- (v) -- (x) -- (w) -- (y);
    \node at (0.2,-0.5){\footnotesize{$3$}};
    \node at (-0.5,0){\footnotesize{$1$}};
    \node at (-0.5,-1){\footnotesize{$2$}};
    \node at (0.6,0){\footnotesize{$2$}};
    \node at (0.6,-1){\footnotesize{$1$}};
    \end{tikzpicture}
    \caption{A diamond graph and its interval coloring with impropriety $1$.}
    \label{fig:diamond}
\end{figure}

If $4 \le \Delta(G) \le 5$, then by Theorem~\ref{thm:delta5}, $\im{G} \le 2 = \left\lceil \frac{\Delta(G)}{3} \right\rceil$. Therefore, $\Delta(G) \ge 6$ for the minimal counterexample $G$.

Let $v$ be a vertex of degree 2 in $G$. Note that $v$ can be considered as the last vertex added to $G$ in the process of obtaining a 2-tree, so $N(v) \cong K_2$.
Let $N(v) = \{u,w\}$ and $uw \in E(G)$.
Since $v$ is the last vertex added to $G$, $G-v$ is also a $2$-tree.
Since $\Delta(G) \ge 6$, we have $\Delta(G-v) \ge 5$ and $\im{G} \le \left\lceil \frac{\Delta(G-v)}{3} \right\rceil$ by the minimality of $G$.
Let $\varphi$ be a $k$-improper interval coloring of $G-v$, where $k \leq \left\lceil \frac{\Delta(G-v)}{3} \right\rceil$.

Now, we can construct the coloring $\varphi'$ of $E(G)$ in the following way. First, for $e \in E(G-v)$, we let $\varphi'(e) = \varphi(e)$. 
Let $x = \varphi(uw)=\varphi'(uw)$. 
Then at least one of the colors in $\{x-1, x, x+1\}$ appears at most $\left\lceil \frac{\Delta(G)}{3} \right\rceil -1$ times on the edges incident to $u$ (and by symmetry, to $w$).
Let $x_u$ (resp. $x_w$) be a color in $\{x-1, x, x+1\}$ that appears at most $\left\lceil \frac{\Delta(G)}{3} \right\rceil -1$ times on the edges incident to $u$ (resp. $w$).

 If $x_u = x$, then assigning $\varphi'(uv) = x_u$ and $\varphi'(wv) = x_w$ provides us a $\left\lceil \frac{\Delta(G)}{3} \right\rceil$-improper interval coloring of $G$, which is a contradiction.
 
 Suppose $x_u = x-1$ (resp. $x_u=x+1$).  
If $x_w = x-1$ (resp. $x_w = x+1$), then assigning $\varphi'(uv) = x_u$ and $\varphi'(wv) = x_w$ provides us a $\left\lceil \frac{\Delta(G)}{3} \right\rceil$-improper interval coloring of $G$, which is a contradiction.
If $x_w \neq x-1$ (resp. $x_w \neq x+1$), then that means $x-1$ (resp. $x+1$) appears on an edge incident with $w$. 
Thus, there is at least one color in $\{x-2, x-1, x\}$ (resp. $\{x+2, x+1, x\}$) appearing at most $\left\lceil \frac{\Delta(G)}{3} \right\rceil -1$ times on the edges incident to $w$.
Let this color be $x'_w$. 
Then assigning $\varphi'(uv) = x_u$ and $\varphi'(wv) = x'_w$ provides us  a $\left\lceil \frac{\Delta(G)}{3} \right\rceil$-improper interval coloring of $G$, which is a contradiction.
This concludes the proof.
\end{proof}

Note that the smallest 2-tree and the only one  with $\Delta(G) = 2$ is $K_3$, and $\im{G} = 2 > 1 = \left \lceil \frac{2}{3} \right \rceil$. Thus the condition that $\Delta(G) \geq 3$ in the preceding theorem is justified. Moreover, we were not able to find any 2-tree with impropriety strictly greater than $2$. On the other hand, there are clearly 2-trees with impropriety 2, for example $K_3$. Thus we pose the following conjecture. 

\begin{conjecture}
    \label{conj:2-trees}
    If G is a 2-tree, then $\mu_{int}(G)\leq 2$.
\end{conjecture}

In the following we provide a better bound for two subfamilies of 2-trees, supporting Conjecture~\ref{conj:2-trees}.

For $n \geq 1$, the \emph{square of a path $P_n$} is the graph G with vertices $V(G)=V(P_n)$ and edges $E(G)=\{uv \colon d_{P_n} (u,v) \leq 2\}$. Observe that the unique 2-tree on 4 vertices (a diamond) is the square of $P_4$. Notice also that if $n \ge 4$, then $G$ contains a vertex $v$ of degree 2, such that $N_G(v) = \{x,y\}$, $xy$ is an edge of $G$, $d_{G-v}(x) = 2$ and $d_{G-v}(y) = 3$. 

\begin{proposition}
\label{prop:2-paths}
If $G$ is the square of a path $P_n$, $n \geq 4$, then $\im{G} = 1$.
\end{proposition}

\begin{proof}
Let $G$ be the minimal counterexample with respect to the number of vertices. 
Then $|V(G)| \ge 5$ since the diamond graph is interval colorable.
Let $v$ be a vertex of degree 2, such that $N_G(v) = \{x,y\}$,
$xy\in E(G)$, $d_{G-v}(x) = 2$ and $d_{G-v}(y) = 3$. 
Notice that $G-v$ is the square of $P_{n-1}$. Then, by the minimality of $G$, $\im{G-v} = 1$. Let $\varphi$ be an interval coloring of $G-v$, and let $\varphi(xy) = a$. By a short case analysis, the coloring of $G-v$ can be extended by appropriately coloring edges $vx$ and $vy$ to give $\im{G} = 1$. See Figure \ref{fig:2-paths} for the analysis of different cases.
\begin{figure}[ht]
\centering
\begin{tikzpicture}
\begin{scope}
    \node[circle, draw=black, fill=white, inner sep =0.05cm, label={90:$y$}](y) at (0,0){};
    \node[circle, draw=black, fill=white, inner sep =0.05cm, label={-90:$x$}](x) at (0,-1){};
    \node[circle, draw=black, fill=black, inner sep =0.05cm, label={0:$v$}](v) at (1,-0.5){};
    \draw (x) -- (y);
    \draw (y) -- (v) -- (x);
    \draw (-1,0.5) -- (y) -- (-1,-0.5);
    \draw (-1,-1.5) -- (x);
    \node at (0.2,-0.5){\footnotesize{$a$}};
    \node at (-0.5,0.6){\footnotesize{$a+2$}};
    \node at (-0.5,-0.6){\footnotesize{$a+1$}};
    \node at (-0.5,-1.6){\footnotesize{$a-1$}};
    \node at (0.6,0){\footnotesize{\boldmath$a-1$}};
    \node at (0.6,-1){\footnotesize{\boldmath$a-2$}};
\end{scope}

\begin{scope}[yshift = 3.5cm]
    \node[circle, draw=black, fill=white, inner sep =0.05cm, label={90:$y$}](y) at (0,0){};
    \node[circle, draw=black, fill=white, inner sep =0.05cm, label={-90:$x$}](x) at (0,-1){};
    \node[circle, draw=black, fill=black, inner sep =0.05cm, label={0:$v$}](v) at (1,-0.5){};
    \draw (x) -- (y);
    \draw (y) -- (v) -- (x);
    \draw (-1,0.5) -- (y) -- (-1,-0.5);
    \draw (-1,-1.5) -- (x);
    \node at (0.2,-0.5){\footnotesize{$a$}};
    \node at (-0.5,0.6){\footnotesize{$a+2$}};
    \node at (-0.5,-0.6){\footnotesize{$a+1$}};
    \node at (-0.5,-1.6){\footnotesize{$a+1$}};
    \node at (0.6,0){\footnotesize{\boldmath$a+3$}};
    \node at (0.6,-1){\footnotesize{\boldmath$a+2$}};
\end{scope}

\begin{scope}[xshift=4cm]
    \node[circle, draw=black, fill=white, inner sep =0.05cm, label={90:$y$}](y) at (0,0){};
    \node[circle, draw=black, fill=white, inner sep =0.05cm, label={-90:$x$}](x) at (0,-1){};
    \node[circle, draw=black, fill=black, inner sep =0.05cm, label={0:$v$}](v) at (1,-0.5){};
    \draw (x) -- (y);
    \draw (y) -- (v) -- (x);
    \draw (-1,0.5) -- (y) -- (-1,-0.5);
    \draw (-1,-1.5) -- (x);
    \node at (0.2,-0.5){\footnotesize{$a$}};
    \node at (-0.5,0.6){\footnotesize{$a+1$}};
    \node at (-0.5,-0.6){\footnotesize{$a-1$}};
    \node at (-0.5,-1.6){\footnotesize{$a-1$}};
    \node at (0.6,0){\footnotesize{\boldmath$a+2$}};
    \node at (0.6,-1){\footnotesize{\boldmath$a+1$}};
\end{scope}

\begin{scope}[xshift=4cm, yshift = 3.5cm]
    \node[circle, draw=black, fill=white, inner sep =0.05cm, label={90:$y$}](y) at (0,0){};
    \node[circle, draw=black, fill=white, inner sep =0.05cm, label={-90:$x$}](x) at (0,-1){};
    \node[circle, draw=black, fill=black, inner sep =0.05cm, label={0:$v$}](v) at (1,-0.5){};
    \draw (x) -- (y);
    \draw (y) -- (v) -- (x);
    \draw (-1,0.5) -- (y) -- (-1,-0.5);
    \draw (-1,-1.5) -- (x);
    \node at (0.2,-0.5){\footnotesize{$a$}};
    \node at (-0.5,0.6){\footnotesize{$a+1$}};
    \node at (-0.5,-0.6){\footnotesize{$a-1$}};
    \node at (-0.5,-1.6){\footnotesize{$a+1$}};
    \node at (0.6,0){\footnotesize{\boldmath$a-2$}};
    \node at (0.6,-1){\footnotesize{\boldmath$a-1$}};
\end{scope}

\begin{scope}[xshift=8cm]
    \node[circle, draw=black, fill=white, inner sep =0.05cm, label={90:$y$}](y) at (0,0){};
    \node[circle, draw=black, fill=white, inner sep =0.05cm, label={-90:$x$}](x) at (0,-1){};
    \node[circle, draw=black, fill=black, inner sep =0.05cm, label={0:$v$}](v) at (1,-0.5){};
    \draw (x) -- (y);
    \draw (y) -- (v) -- (x);
    \draw (-1,0.5) -- (y) -- (-1,-0.5);
    \draw (-1,-1.5) -- (x);
    \node at (0.2,-0.5){\footnotesize{$a$}};
    \node at (-0.5,0.6){\footnotesize{$a-1$}};
    \node at (-0.5,-0.6){\footnotesize{$a-2$}};
    \node at (-0.5,-1.6){\footnotesize{$a-1$}};
    \node at (0.6,0){\footnotesize{\boldmath$a-3$}};
    \node at (0.6,-1){\footnotesize{\boldmath$a-2$}};
\end{scope}

\begin{scope}[xshift=8cm, yshift = 3.5cm]
    \node[circle, draw=black, fill=white, inner sep =0.05cm, label={90:$y$}](y) at (0,0){};
    \node[circle, draw=black, fill=white, inner sep =0.05cm, label={-90:$x$}](x) at (0,-1){};
    \node[circle, draw=black, fill=black, inner sep =0.05cm, label={0:$v$}](v) at (1,-0.5){};
    \draw (x) -- (y);
    \draw (y) -- (v) -- (x);
    \draw (-1,0.5) -- (y) -- (-1,-0.5);
    \draw (-1,-1.5) -- (x);
    \node at (0.2,-0.5){\footnotesize{$a$}};
    \node at (-0.5,0.6){\footnotesize{$a-1$}};
    \node at (-0.5,-0.6){\footnotesize{$a-2$}};
    \node at (-0.5,-1.6){\footnotesize{$a+1$}};
    \node at (0.6,0){\footnotesize{\boldmath$a+1$}};
    \node at (0.6,-1){\footnotesize{\boldmath$a+2$}};
\end{scope}

\end{tikzpicture}
        \caption{Case analysis for Proposition \ref{prop:2-paths}.}
        \label{fig:2-paths}
    \end{figure}
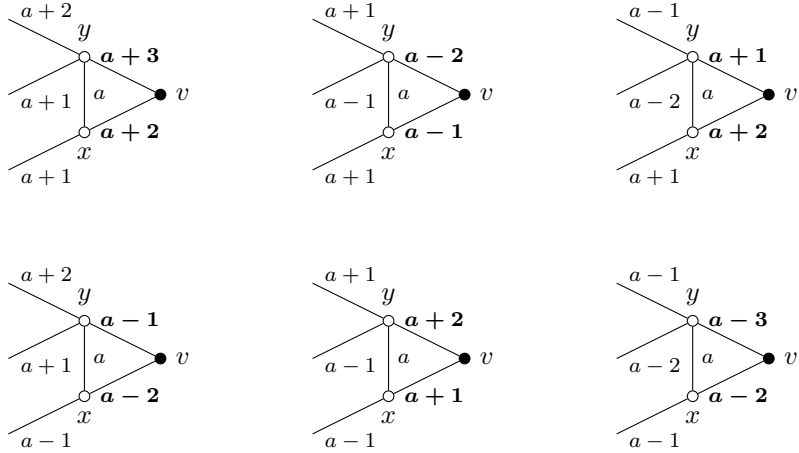
\end{proof}

Note that squares of paths are a special case of 2-paths \cite{k-trees-2006, k-trees-1984}. A \emph{2-path} is an alternating sequence of distinct edges and triangles $(e_0, t_1, e_1, t_2, \ldots, e_n)$, starting and ending with an edge, such that $t_i$ contains exactly two of the distinct edges in the sequence, $e_{i-1}$ and $e_i$. For example, see Figure \ref{fig:2-path-example}.

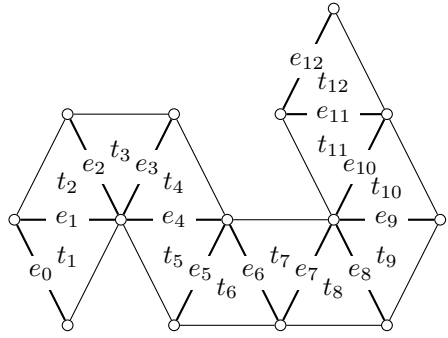
\begin{figure}
    \centering
    \begin{tikzpicture}[scale=0.7]
        \node[circle, draw=black, fill=white, inner sep =0.05cm](a) at (0,0){};
        \node[circle, draw=black, fill=white, inner sep =0.05cm](b) at (1,2){};
        \node[circle, draw=black, fill=white, inner sep =0.05cm](c) at (-1,2){};
        \node[circle, draw=black, fill=white, inner sep =0.05cm](d) at (0,4){};
        \node[circle, draw=black, fill=white, inner sep =0.05cm](e) at (2,4){};
        \node[circle, draw=black, fill=white, inner sep =0.05cm](f) at (3,2){};
        \node[circle, draw=black, fill=white, inner sep =0.05cm](g) at (2,0){};
        \node[circle, draw=black, fill=white, inner sep =0.05cm](h) at (4,0){};
        \node[circle, draw=black, fill=white, inner sep =0.05cm](i) at (5,2){};
        \node[circle, draw=black, fill=white, inner sep =0.05cm](j) at (6,0){};
        \node[circle, draw=black, fill=white, inner sep =0.05cm](k) at (7,2){};
        \node[circle, draw=black, fill=white, inner sep =0.05cm](l) at (6,4){};
        \node[circle, draw=black, fill=white, inner sep =0.05cm](m) at (4,4){};
        \node[circle, draw=black, fill=white, inner sep =0.05cm](n) at (5,6){};

        \draw[thick] (a) -- (c) node [midway, fill=white] {$e_0$};
        \draw[thick] (c) -- (b) node [midway, fill=white] {$e_1$};
        \draw[thick] (d) -- (b) node [midway, fill=white] {$e_2$};
        \draw[thick] (e) -- (b) node [midway, fill=white] {$e_3$};
        \draw[thick] (f) -- (b) node [midway, fill=white] {$e_4$};
        \draw[thick] (f) -- (g) node [midway, fill=white] {$e_5$};
        \draw[thick] (h) -- (f) node [midway, fill=white] {$e_6$};
        \draw[thick] (h) -- (i) node [midway, fill=white] {$e_7$};
        \draw[thick] (i) -- (j) node [midway, fill=white] {$e_8$};
        \draw[thick] (i) -- (k) node [midway, fill=white] {$e_9$};
        \draw[thick] (i) -- (l) node [midway, fill=white] {$e_{10}$};
        \draw[thick] (l) -- (m) node [midway, fill=white] {$e_{11}$};
        \draw[thick] (n) -- (m) node [midway, fill=white] {$e_{12}$};

        \draw (a) -- (b);
        \draw (c) -- (d);
        \draw (d) -- (e);
        \draw (e) -- (f);
        \draw (b) -- (g);
        \draw (g) -- (h);
        \draw (f) -- (i);
        \draw (h) -- (j);
        \draw (j) -- (k);
        \draw (k) -- (l);
        \draw (l) -- (n);
        \draw (i) -- (m);

        \node (1) at (0,1.3){$t_1$};
        \node (2) at (0,2.7){$t_2$};
        \node (3) at (1,3.3){$t_3$};
        \node (4) at (2,2.7){$t_4$};
        \node (5) at (2,1.3){$t_5$};
        \node (6) at (3,0.7){$t_6$};
        \node (7) at (4,1.3){$t_7$};
        \node (8) at (5,0.7){$t_8$};
        \node (9) at (6,1.3){$t_9$};
        \node (10) at (6,2.6){$t_{10}$};
        \node (11) at (5,3.4){$t_{11}$};
        \node (12) at (5,4.6){$t_{12}$};

    \end{tikzpicture}
    \caption{An example of a $2$-path.}
    \label{fig:2-path-example}
\end{figure}

\begin{proposition}
\label{prop:2-paths-real}
If $G$ is a 2-path, then $\im{G} \leq 2$. 
\end{proposition}

\begin{proof}
    Let $(e_0, t_1, e_1, t_2, \ldots, e_n)$ be the sequence defining $G$. Define an edge-coloring $\varphi \colon E(G) \to \mathbb{N}$ in the following way. For every $i \in \{0\} \cup [n]$, let $\varphi(e_i) = i$, and for every $i \in [n]$, color the remaining edge of $t_i$ with color $i$.

    Let $v \in V(G)$. By definition of a 2-path, there exist indices $0 \leq a \leq b \leq n$ such that $v$ is incident to $e_i$ if and only if $a \leq i \leq b$. Thus $v$ is incident to a triangle $t_j$ if and only if $\max\{a, 1\} \leq j \leq \min\{b + 1, n\}$. Hence, the colors on edges incident to $v$ are $a, \ldots, b$ and $\max\{a, 1\}, \ldots, \min\{b + 1, n\}$.
    This is an interval and each color occurs at most twice. Therefore, $\im{G} \leq 2$.
\end{proof}

\subsection{Iterated triangulations}

In this section we consider a subfamily of 3-trees. The $n$-th \emph{iterated triangulation} $\Tr(n)$ is the plane graph defined as follows. The graph $K_3$ is $\Tr(0)$.
For each $i \ge 0$, obtain $\Tr(i+1)$ from $\Tr(i)$ by adding a new vertex $v_f$ to the interior of each inner face $f$ of $\Tr(i)$ and connecting $v_f$ with each of the three vertices on $f$. 
For example, see Figure~\ref{fig:triangulation}.
Note that $\Tr(n)$ is a planar graph, and the boundary of each face of $\Tr(n)$ is a $3$-cycle.
In addition, $\Tr(n)$ is a special case of the Apollonian network, sometimes called the {\it complete Apollonian network}. 
Recall that Apollonian networks are precisely the planar 3-trees. 

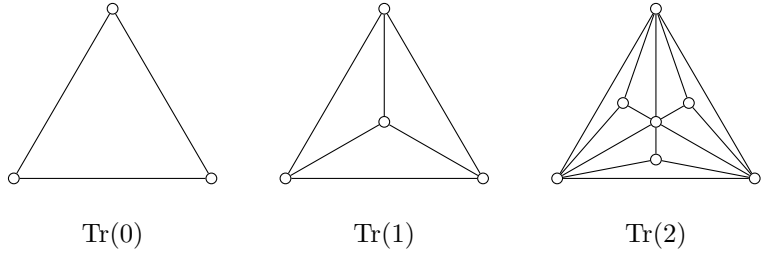
\begin{figure}
\centering
\begin{tikzpicture}
\node[circle,fill=white, draw=black, inner sep=0.05cm] (n0) at (-30:1.5){};
\node[circle,fill=white, draw=black, inner sep=0.05cm] (n1) at (90:1.5){};
\node[circle,fill=white, draw=black, inner sep=0.05cm] (n2) at (210:1.5){};
\draw (n0) -- (n1) -- (n2) -- (n0);
\node at (0,-1.5) {$\Tr(0)$};
\end{tikzpicture}
\qquad
\begin{tikzpicture}
\node[circle,fill=white, draw=black, inner sep=0.05cm] (n0) at (-30:1.5){};
\node[circle,fill=white, draw=black, inner sep=0.05cm] (n1) at (90:1.5){};
\node[circle,fill=white, draw=black, inner sep=0.05cm] (n2) at (210:1.5){};
\node[circle,fill=white, draw=black, inner sep=0.05cm] (m0) at (0,0){};
\draw (n0) -- (n1) -- (n2) -- (n0);
\draw (n0) -- (m0);
\draw (n1) -- (m0);
\draw (n2) -- (m0);
\node at (0,-1.5) {$\Tr(1)$};
\end{tikzpicture}
\qquad
\begin{tikzpicture}
\node[circle,fill=white, draw=black, inner sep=0.05cm] (n0) at (-30:1.5){};
\node[circle,fill=white, draw=black, inner sep=0.05cm] (n1) at (90:1.5){};
\node[circle,fill=white, draw=black, inner sep=0.05cm] (n2) at (210:1.5){};
\draw (n0) -- (n1) -- (n2) -- (n0);
\node[circle,fill=white, draw=black, inner sep=0.05cm] (m0) at (0,0){};
\draw (n0) -- (m0);
\draw (n1) -- (m0);
\draw (n2) -- (m0);
\node[circle,fill=white, draw=black, inner sep=0.05cm] (p0) at (30:0.5){};
\node[circle,fill=white, draw=black, inner sep=0.05cm] (p1) at (150:0.5){};
\node[circle,fill=white, draw=black, inner sep=0.05cm] (p2) at (-90:0.5){};
\draw (p0) -- (m0);
\draw (p0) -- (n0);
\draw (p0) -- (n1);
\draw (p1) -- (m0);
\draw (p1) -- (n1);
\draw (p1) -- (n2);
\draw (p2) -- (m0);
\draw (p2) -- (n0);
\draw (p2) -- (n2);
\node at (0,-1.5) {$\Tr(2)$};
\end{tikzpicture}
    \caption{Illustrations for $\Tr(0)$, $\Tr(1)$, and $\Tr(2)$.}
    \label{fig:triangulation}
\end{figure}

\begin{theorem}
For every $n \geq 1$, $\im{\Tr(n)} \le \left\lceil \frac{\Delta(\Tr(n))}{3} \right\rceil$.
\end{theorem}
\begin{proof}
By observing a construction of $\Tr(n)$, we see that $\Tr(n)$ contains $\Tr(i)$ for all $i \in \{0\} \cup [n]$.
For $i \in [n]$, let us call the vertices in $V(\Tr(i)) \setminus V(\Tr(i-1)) \subseteq V(\Tr(n))$, the {\em $i$-th vertices} in $\Tr(n)$.

Let $C = v_1 v_2 v_3 v_1$ be a $3$-cycle that is the boundary of the outer face of $\Tr(n)$. 
Then $C$ is a $\Tr(0)$.
Let $\varphi$ be the edge coloring of $C$ defined by $\varphi(v_1v_2)=1$, $\varphi(v_2v_3)=2$, and $\varphi(v_3v_1)=3$.

For each $i \in [n]$, given the edge coloring $\varphi$ of $\Tr(i-1)$, we color the edges of $\Tr(i)$ by giving colors as follows.
Let $v$ be a $i$-th vertex in $\Tr(n)$.
Then $v$ is located inside the face $f$ of $\Tr(i-1)$.
Let $C_f = u_1 u_2 u_3 u_1$ be the boundary 
of the face $f$.
Then for each $j \in [3]$, we let $\varphi(vu_j) \in [3] \setminus \{\varphi(u_ju_{j+1}), \varphi(u_ju_{j-1})\}$, where the addition in the subscript is modulo $3$.

Then for each $v \in V(\Tr(n))$, edges incident with $v$ will receive all three colors in $[3]$ appearing periodically, in the counterclockwise direction.
For example, the edges incident with $v$ will receive $x$, $y$, $z$, $x$, $y$, $z$, $\ldots$, where $\{x,y,z\} = [3]$.
Therefore, $\varphi$ is a $k$-improper interval coloring of $\Tr(n)$, with $k \leq \left\lceil \frac{\Delta(\Tr(n))}{3} \right\rceil$.
\end{proof}

Note that the theorem does not hold for $Tr(0)$, since $Tr(0) = K_3$ and thus has impropriety $2 > \left\lceil \frac{2}{3} \right\rceil$.

\section{Outerplanar graphs}

An outerplanar graph is Class 1 unless it is an odd cycle. 
In \cite{casselgren2021improper}, Casselgren and Petrosyan proved the following theorem on the impropriety of an outerplanar graph with respect to the maximum degree. 

\begin{theorem}[\cite{casselgren2021improper}]\label{thm:D4}
    If $G$ is an outerplanar graph, then $\mu_{int}(G) \le \left\lceil \frac{\Delta(G)}{4}\right\rceil +1$.
\end{theorem}

They also made the following conjecture, and proved that the Conjecture~\ref{conj:outerplanar} is true when $\Delta(G) \le 8$.

\begin{conjecture}[\cite{casselgren2021improper}]\label{conj:outerplanar}
If $G$ is an outerplanar graph, then $\mu_{int}(G) \le 2$.   
\end{conjecture}

This conjecture is sharp if true since every odd cycle has impropriety exactly $2$.
In this section, we improve Theorem~\ref{thm:D4}. Note that this result proves that Conjecture~\ref{conj:outerplanar} is true when $\Delta(G) \le 10$. We will make use of the following lemma.

\begin{lemma}\label{lem:2con}
     Suppose $G$ is a $2$-connected outerplanar graph such that every $2$-vertex in $G$ is on a $3$-cycle. 
Then there is a $2$-vertex in $G$ that has a neighbor of degree at most $4$.
   
\end{lemma}

\begin{proof} 
Assume that $G$ is an outerplanar graph not satisfying Lemma~\ref{lem:2con}. Then each neighbour of a 2-vertex has degree at least 5. Denote by $S$ the set of all 2-vertices. Then, every other vertex in the graph may have at most two neighbors from $S$. Thus, the graph $G\setminus S$ is an outerplanar graph of minimum degree 3, a contradiction. 

\end{proof}

\begin{theorem}\label{thm:D5}
    For any outerplanar graph $G$ with $\Delta(G) \ge 6$, $\mu_{int}(G) \le \left\lceil \frac{\Delta(G)}{5}\right\rceil$.
\end{theorem}

\begin{proof}
Let $G$ be a minimal counterexample to the statement with respect to the number of vertices. So $G$ is an outerplanar graph with $\Delta(G) \geq 6$ and $\mu_{int}(G)> \left\lceil \frac{\Delta(G)}{5}\right\rceil$.

\begin{claim}
    $G$ is $2$-connected.    
\end{claim}
\begin{myproof}
    Suppose on the contrary that $G$ not $2$-connected.
    
    If $G$ is disconnected, and $G$ is a disjoint union of two subgraphs $G_1$ and $G_2$, then by the minimality of $G$, $\mu_{int}(G_1) \le \left\lceil \frac{\Delta(G)}{5}\right\rceil$ and $\mu_{int}(G_2) \le \left\lceil \frac{\Delta(G)}{5}\right\rceil$. This implies $\mu_{int}(G) \le \left\lceil \frac{\Delta(G)}{5}\right\rceil$, which is a contradiction.
    Thus, $G$ is connected.

    Let $v$ be a cut vertex of $G$, and let $G_1$ and $G_2$ be two subgraphs of $G$ such that $V(G_1) \cap V(G_2) = \{v\}$.
    By the minimality of $G$, $\mu_{int}(G_1) \le \left\lceil \frac{\Delta(G)}{5}\right\rceil$ and $\mu_{int}(G_2) \le \left\lceil \frac{\Delta(G)}{5}\right\rceil$. 
    We shift the colors of $G_2$ appropriately so that the colors on the edges incident with $v$ form an interval.
    This implies that $\mu_{int}(G) \le \left\lceil \frac{\Delta(G)}{5}\right\rceil$, which is a contradiction.
\end{myproof}

Since $G$ is a $2$-connected outerplanar graph, by Lemma~\ref{lem:2con}, $G$ has a $2$-vertex $v$ that is either
\begin{itemize}
        \item[(i)] on a $3$-cycle $vuw$ such that $d_G(u) \le 4$, or; 
        \item[(ii)] not on a $3$-cycle.
    \end{itemize}

\textbf{Case (1)} Suppose (i) holds. 
Then $G$ has a $2$-vertex that is on a $3$-cycle $vuw$ such that $d_G(u) \le 4$.    
By the minimality of $G$, $G-v$ has an $\left\lceil \frac{\Delta(G)}{5} \right\rceil$-improper interval coloring $\varphi$.
Since $\Delta(G) \ge 6$, $\im{G} > \left\lceil \frac{\Delta(G)}{5} \right\rceil \ge 2$.
Let $\varphi(uw) = x$.

Since $d_{G-v}(u) \le 3$, we may assign $\varphi(vu) \in \{x, x+1\}$ so that the interval property at $u$ is preserved and the impropriety at $u$ is still at most $\left\lceil \frac{\Delta(G)}{5} \right\rceil$.
If assigning $\varphi(vw)=x$ does not make the impropriety at $w$ greater than $\left\lceil \frac{\Delta(G)}{5} \right\rceil$, then it is a contradiction that $\im{G} > \left\lceil \frac{\Delta(G)}{5} \right\rceil$.
Thus, assigning $\varphi(vw)=x$ makes the impropriety at $w$ greater than $\left\lceil \frac{\Delta(G)}{5} \right\rceil$, and this means that $w$ is incident with at least $\left\lceil \frac{\Delta(G)}{5}\right\rceil$ edges in $G-v$ that are colored $x$ by $\varphi$.
Similarly, $w$ is incident with at least $\left\lceil \frac{\Delta(G)}{5}\right\rceil$ 
edges in $G-v$ that are colored $x+1$ by $\varphi$.

Since $d_{G-v}(u) \le 3$, we may assign $\varphi(vu) \in \{x+1, x+2\}$ so that the interval property at $u$ is preserved and the impropriety at $u$ is still at most $\left\lceil \frac{\Delta(G)}{5} \right\rceil$.
If assigning $\varphi(vw)=x+2$ does not make the impropriety at $w$ greater than $\left\lceil \frac{\Delta(G)}{5} \right\rceil$, then it is a contradiction that $\im{G} > \left\lceil \frac{\Delta(G)}{5} \right\rceil$.
Thus, assigning $\varphi(vw)=x+2$ makes the impropriety at $w$ greater than $\left\lceil \frac{\Delta(G)}{5} \right\rceil$, and this means that $w$ is incident with at least $\left\lceil \frac{\Delta(G)}{5}\right\rceil$ edges in $G-v$ that are colored $x+2$ by $\varphi$.

Since $d_{G-v}(u) \le 3$, we may assign $\varphi(vu) \in \{x, x-1\}$ so that the interval property at $u$ is preserved and the impropriety at $u$ is still at most $\left\lceil \frac{\Delta(G)}{5} \right\rceil$.
If assigning $\varphi(vw)=x-1$ does not make the impropriety at $w$ greater than $\left\lceil \frac{\Delta(G)}{5} \right\rceil$, then it is a contradiction that $\im{G} > \left\lceil \frac{\Delta(G)}{5} \right\rceil$.
Thus, assigning $\varphi(vw)=x-1$ makes the impropriety at $w$ greater than $\left\lceil \frac{\Delta(G)}{5} \right\rceil$, and this means that $w$ is incident with at least $\left\lceil \frac{\Delta(G)}{5}\right\rceil$ edges in $G-v$ that are colored $x-1$ by $\varphi$.

By the above arguments, $w$ is incident with at least $4 \cdot \left\lceil \frac{\Delta(G)}{5}\right\rceil$ edges in $G-v$ that are colored by a color in $\{x-1,x,x+1,x+2\}$.
Since there are at most $\Delta(G)-1$ edges in $G-v$ that are incident with $w$, $w$ is incident with less than $\left\lceil \frac{\Delta(G)}{5} \right\rceil$ edges in $G-v$ colored $x-2$ by $\varphi$. 
Thus, we may assign $\varphi(vu) \in \{x-1, x-2\}$ and $\varphi(vw)=x-2$, and obtain an $\left\lceil \frac{\Delta(G)}{5} \right\rceil$-improper interval coloring $\varphi$ of $G$, which is a contradiction.

\medskip

\textbf{Case (2)} Suppose (ii) holds. Then $G$ has a $2$-vertex $v$ that is not on a $3$-cycle.
Let $N_G(v)=\{u,w\}$.
Since $v$ is not on a $3$-cycle, $uw \not\in E(G)$.
By the minimality of $G$, $G-v+uw$ has an $\left\lceil \frac{\Delta(G)}{5} \right\rceil$-improper interval coloring $\varphi$.
Let $\varphi(uw)=x$. 
Then by assigning $\varphi(vu)=\varphi(vw)=x$, $\varphi$ becomes an $\left\lceil \frac{\Delta(G)}{5} \right\rceil$-improper interval coloring of $G$, which is a contradiction.
\end{proof}

\section{Corona Products}

Let $G$ and $H$ be graphs, with $V(G) = \{v_1,v_2,\dots,v_n\}$. The \emph{corona product} of $G$ and $H$, denoted by $G \odot H$, is the graph obtained by taking $G$ and $n$ disjoint copies of $H$, say $H_1,H_2,\dots,H_n$, and for each $i \in [n]$, adding the edge $v_i x$ for every $x\in V(H_i)$. Note that $K_1 \odot H$ is obtained from $H$ by adding a universal vertex to $H$. See Figure~\ref{fig:c4p3} for an example of the corona product $C_4 \odot P_3$.

\begin{figure}[ht]
\centering
    \begin{tikzpicture}[every node/.style={circle, draw=black, fill=white, inner sep=0.05cm}]

 		\node (0) at (-0.65*1, 0.65*1) {};
		\node (1) at (0.65*1, 0.65*1) {};
		\node (2) at (0.65*1, -0.65*1) {};
		\node (3) at (-0.65*1, -0.65*1) {};
		\node (4) at (-0.65*2, 0.65*2) {};
		\node (5) at (-0.65*1, 0.65*3) {};
		\node (6) at (-0.65*3, 0.65*1) {};
		\node (7) at (0.65*1, 0.65*3) {};
		\node (8) at (0.65*2, 0.65*2) {};
		\node (9) at (0.65*3, 0.65*1) {};
		\node (10) at (0.65*1, -0.65*3) {};
		\node (11) at (0.65*2, -0.65*2) {};
		\node (12) at (0.65*3, -0.65*1) {};
		\node (13) at (-0.65*2, -0.65*2) {};
		\node (14) at (-0.65*3, -0.65*1) {};
		\node (15) at (-0.65*1, -0.65*3) {};

 		\draw[ultra thick] (0) to (1);
		\draw[ultra thick] (1) to (2);
		\draw[ultra thick] (2) to (3);
		\draw[ultra thick] (3) to (0);
		\draw (0) to (4);
		\draw (0) to (5);
		\draw (0) to (6);
		\draw (1) to (7);
		\draw (1) to (8);
		\draw (1) to (9);
		\draw (2) to (12);
		\draw (2) to (11);
		\draw (2) to (10);
		\draw (3) to (15);
		\draw (3) to (13);
		\draw (3) to (14);
		\draw[ultra thick] (6) to (4);
		\draw[ultra thick] (4) to (5);
		\draw[ultra thick] (7) to (8);
		\draw[ultra thick] (8) to (9);
		\draw[ultra thick] (12) to (11);
		\draw[ultra thick] (11) to (10);
		\draw[ultra thick] (15) to (13);
		\draw[ultra thick] (13) to (14);
\end{tikzpicture}
\caption{The corona product $C_4 \odot P_3$.}
\label{fig:c4p3}
\end{figure}
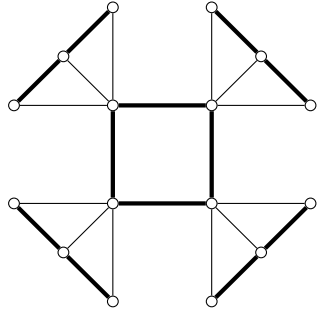

We first determine an upper bound on the impropriety of some corona products of the form $K_1 \odot H$, where the graph $H$ belongs to one of several well-known classes of graphs. After establishing this result, we expand it to an upper bound on the impropriety of $G \odot H$, where $G$ is any arbitrary graph. 

\begin{theorem}\label{thm:universal}
Let $H$ be a graph that is one of the following:
\begin{itemize}
\item[(1)] a path $P_k$, 
\item[(2)] a cycle $C_k$, 
\item[(3)] a star $S_k$, 
\item[(4)] a spider $SP_k$ ($k \le 4$).
\end{itemize}
Then $\im{K_1\odot H} \le 2$.
\end{theorem}

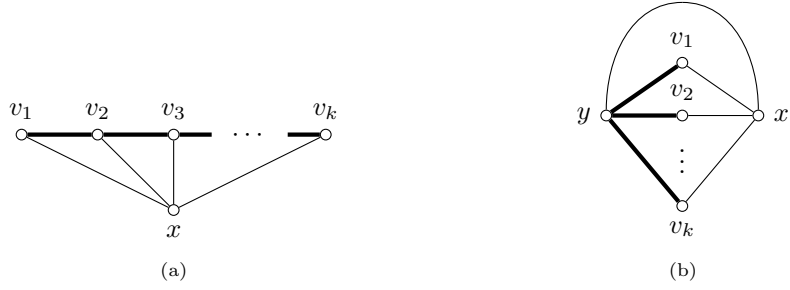
\begin{figure}
    \centering
    \begin{subfigure}[b]{0.4\textwidth}
    \centering
\begin{tikzpicture}
[vert/.style={circle, fill=white, draw=black, inner sep=0.05cm}]
\begin{scope}
\node[vert, label={-90:$x$}] (X) at (0,0) {};
\node[vert, label={90:$v_1$}] (V1) at (-2,1) {};
\node[vert, label={90:$v_2$}] (V2) at (-1,1) {};
\node[vert, label={90:$v_3$}] (V3) at (0,1) {};
\node[vert, label={90:$v_k$}] (Vk) at (2,1) {};

\draw (X) -- (V1);
\draw (X) -- (V2);
\draw (X) -- (V3);
\draw (X) -- (Vk);
\draw[ultra thick] (V1) -- (V2);
\draw[ultra thick] (V2) -- (V3);
\draw[ultra thick] (V3) -- (0.5,1);
\draw[ultra thick] (1.5,1) -- (Vk);

\node at (1,1) {$\cdots$};
\end{scope}
\end{tikzpicture}
\caption{}
\label{fig:universal_a}
\end{subfigure}
\begin{subfigure}[b]{0.4\textwidth}
\centering
\begin{tikzpicture}
[vert/.style={circle, fill=white, draw=black, inner sep=0.05cm}]

\begin{scope}
\node[vert, label={180:$y$}] (Y) at (0,0){};
\node[vert, label={90:$v_1$}] (V1) at (1,0.7){};
\node[vert, label={90:$v_2$}] (V2) at (1,0){};
\node at (1,-0.5) {$\vdots$};
\node[vert, label={-90:$v_k$}] (Vk) at (1,-1.2){};
\node[vert, label={0:$x$}] (X) at (2,0){};
\draw[ultra thick] (Y) -- (V1)
(Y) -- (V2)
(Y) -- (Vk);
\draw (X) -- (V1)
(X) -- (V2)
(X) -- (Vk);
\draw (Y) to[out=90, in=180] (1,1.5);
\draw (X) to[out=90, in=0] (1,1.5);
\end{scope}

\end{tikzpicture}
\caption{}
\label{fig:universal_b}
\end{subfigure}
\caption{Illustrations of $G = K_1 \odot H$, where $H$ is path on $k$ vertices or $H$ is a star with $k$ leaves.}
    \label{fig:universal}
\end{figure}

\begin{proof}
(1) Let $H$ be a path $v_1 v_2 \ldots v_k$.
Let $G = K_1 \odot H$ be obtained by attaching a universal vertex $x$ to every vertex of $H$.
See Figure~\ref{fig:universal_a}.
Let us define $\varphi:E(G) \rightarrow \mathbb{N}$ by the following rule.
For each $i \in [k-1]$, let $\varphi(v_ix) = \varphi(v_i v_{i+1}) = i$.
Let $\varphi(v_{k}x) = k$.
Then $\varphi$ is a 2-improper interval coloring of $G$.
Thus, $\im{G} \le 2$.

(2) If $H$ is a cycle, then $K_1 \odot H$ is a wheel graph, which has impropriety at most 2 by Theorem~\ref{thm:wheels}. 
Thus, $\im{G} \le 2$.

(3) Let $H$ be a star with its center vertex $y$ and leaves $v_1, v_2, \ldots, v_k$.
Let $G= K_1 \odot H$ be obtained by attaching a universal vertex $x$ to every vertex of $H$.
See Figure~\ref{fig:universal_b}.
Let us define $\varphi:E(G) \rightarrow \mathbb{N}$ by the following rule.
For each $i \in [k]$, let $\varphi(yv_i) = \varphi(xv_i) = i$.
Let $\varphi(xy) = k+1$.
Then $\varphi$ is a 2-improper interval coloring of $G$.
Thus, $\im{G} \le 2$.

(4) Let $H$ be a spider graph with center $v$ and four legs $(u_1, u_2, \ldots, u_{k})$, $(w_1, w_2, \ldots, w_{\ell})$, $(y_1, y_2, \ldots, y_{s})$, and $(z_1, z_2, \ldots, z_{t})$. 
Let $G = K_1 \odot H$ be obtained by attaching a universal vertex $x$ to every vertex of $H$.
See Figure~\ref{fig:universal2}.
Let us define $\varphi:E(G) \rightarrow \mathbb{Z}$ by the following rule.
Let $\varphi(xv) = 0$, $\varphi(vu_1) = \varphi(xu_1) = \varphi(vw_1) = \varphi (xw_1)=1$, and $\varphi(vy_1) = \varphi(xy_1) = \varphi(vz_1) = \varphi(xz_1) = -1$.
For each $i \in [k-1]$, 
let $\varphi(u_{i+1}u_i) = \varphi(u_{i+1}x)=i+1$.
For each $i \in [\ell-1]$,
let $\varphi(w_{i+1}w_i) = \varphi(w_{i+1}x)=i+1$.
For each $i \in [s-1]$,
let $\varphi(y_{i+1}y_i) = \varphi(y_{i+1}x)=-(i+1)$. 
For each $i \in [t-1]$,
let $\varphi(z_{i+1}z_i) = \varphi(z_{i+1}x)=-(i+1)$. 
Then $\varphi$ is a 2-improper interval coloring of $G$.
Thus, $\im{G} \le 2$.

\begin{figure}[ht]
\centering
\begin{tikzpicture}
[vert/.style={circle, fill=white, draw=black, inner sep=0.05cm}]
\node[vert, label={90:$v$}] (V) at (0:0){};
\node[vert, label={90:$u_1$}](U1) at (25:0.9){};
\node[vert, label={90:$u_2$}](U2) at (25:2.75){};
\node at (25:3.55) {.};
\node at (25:3.75) {.};
\node at (25:3.95) {.};
\node[vert, label={90:$u_k$}](Uk) at (25:4.75){};

\node[vert, label={90:$w_1$}](W1) at (155:0.9){};
\node[vert, label={90:$w_2$}](W2) at (155:2.75){};
\node at (155:3.55) {.};
\node at (155:3.75) {.};
\node at (155:3.95) {.};
\node[vert, label={90:$w_{\ell}$}](Wl) at (155:4.75){};

\node[vert, label={90:$y_1$}](Y1) at (-150:1){};
\node[vert, label={90:$y_2$}](Y2) at (-150:2){};
\node at (-150:3.2) {.};
\node at (-150:3) {.};
\node at (-150:2.8) {.};
\node[vert, label={90:$y_s$}](Ys) at (-150:4){};

\node[vert, label={90:$z_1$}](Z1) at (-30:1){};
\node[vert, label={90:$z_2$}](Z2) at (-30:2){};
\node at (-30:3.2) {.};
\node at (-30:3) {.};
\node at (-30:2.8) {.};
\node[vert, label={90:$z_t$}](Zt) at (-30:4){};

\node[vert, label={-90:$x$}](X) at (-90:2){};

\draw[ultra thick] (155:3.15) -- (W2)
(-30:2.5) -- (Z2)
(25:3.15) -- (U2)
(-150:2.5) -- (Y2)
(Wl) -- (155:4.25)
(Ys) -- (-150:3.5)
(Uk) -- (25:4.25)
(Zt) -- (-30:3.5);

\draw[ultra thick] (V) -- (W1)
(V) -- (Y1)
(V) -- (U1)
(V) -- (Z1)
(W1) -- (W2)
(U1) -- (U2)
(Y1) -- (Y2)
(Z1) -- (Z2);

\draw (X) -- (V)
(X) -- (Y1)
(X) -- (Y2)
(X) -- (Ys)
(X) -- (Z1)
(X) -- (Z2)
(X) -- (Zt)
(X) -- (W1)
(X) -- (W2) 
(X) -- (Wl)
(X) -- (U1)
(X) -- (U2)
(X) -- (Uk)
;
\end{tikzpicture}
\caption{Illustration of $G = K_1 \odot H$, where $H$ is a spider graph with four legs.}
\label{fig:universal2}
\end{figure}
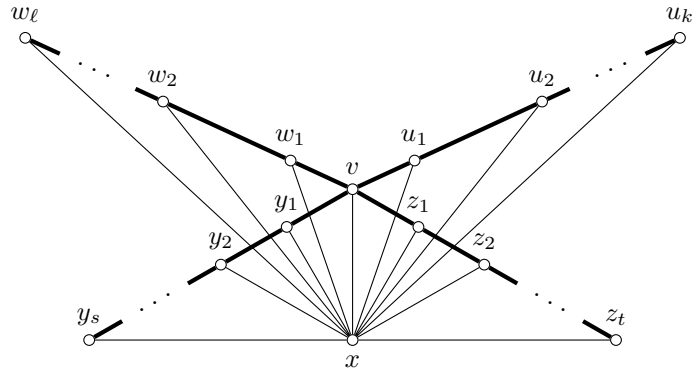

The case for a spider graph with less than four legs follows analogously from the above proof.
\end{proof}

As a corollary to Theorem~\ref{thm:universal}, we obtain the following result.

\begin{corollary}\label{cor:corona_product}
Let $H$ be a graph that is one of the following:
\begin{itemize}
\item[(1)] a path $P_k$, 
\item[(2)] a cycle $C_k$, 
\item[(3)] a star $S_k$, 
\item[(4)] a spider $SP_k$ ($k \le 4$).
\end{itemize}
Then for any graph $G$,
$\im{G \odot H} \leq \max\{2,\im{G}\}$.
\end{corollary}
\begin{proof}
Note that $G \odot H$ is obtained by identifying each vertex of $G$ 
with a universal vertex of a copy of $K_1 \odot H$.
Let $\varphi$ be a $\mu_{int}(G)$-improper interval coloring of $G$.
By Theorem~\ref{thm:universal}, for each $v \in V(G)$, we can extend $\varphi$ into a copy of $K_1 \odot H$ with impropriety at most $2$. 
This can be done by increasing the colors used in the improper interval coloring of each copy of $K_1 \odot H$ so that the minimum color assigned to an edge incident with the universal vertex is one greater than the maximum color assigned by $\varphi$ to an edge in $E(G)$ incident with the corresponding vertex in 
$G$.
Thus, $\im{G \odot H} \leq \max\{2,\im{G}\}$.
\end{proof}

We present an example of a corona product of two graphs that attains this upper bound. 

\begin{example}
\label{ex:corona}
    \emph{Consider the graph $G=C_3 \odot P_2$. Let $u,v,w$ be the vertices of $C_3$ and $a_i,b_i$ the vertices of each copy of $P_2$, for $i\in\{u,v,w\}$. See Figure~\ref{fig:c3p2}. We will show that there does not exist an interval coloring of $G$, and then construct a $2$-improper interval coloring. Observe that $2 = \max\{2, \im{C_3}\}$.}
    
    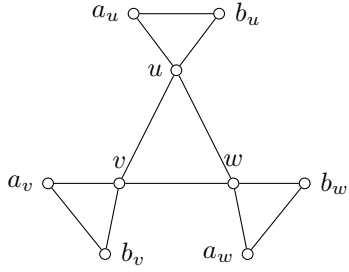
\begin{figure}[ht]
\centering
    \begin{tikzpicture}[every node/.style={circle, draw=black, fill=white, inner sep=0.05cm}]

 	\node [label=left:$u$] (0) at (0, 0.75) {};
		\node [label=$v$] (1) at (-0.75, -0.75) {};
		\node [label=$w$] (2) at (0.75, -0.75) {};
		\node [label=left:$a_u$] (3) at (-0.5625, 1.5) {};
		\node [label=right:$b_u$] (4) at (0.5625, 1.5) {};
		\node [label=left:$a_v$] (5) at (-1.6875, -0.75) {};
		\node [label=right:$b_v$] (6) at (-0.9375, -1.6875) {};
		\node [label=left:$a_w$] (7) at (0.9375, -1.6875) {};
		\node [label=right:$b_w$] (8) at (1.6875, -0.75) {};

  \draw (0) to (1);
		\draw (1) to (2);
		\draw (2) to (0);
		\draw (0) to (3);
		\draw (3) to (4);
		\draw (4) to (0);
		\draw (5) to (1);
		\draw (1) to (6);
		\draw (6) to (5);
		\draw (2) to (7);
		\draw (7) to (8);
		\draw (8) to (2);
\end{tikzpicture}
\caption{The corona product $C_3 \odot P_2$ in Example~\ref{ex:corona}.}
\label{fig:c3p2}
\end{figure}
    
    \emph{Suppose, without loss of generality, that the edge $a_ub_u$ receives color $x$. Then the edges $ua_u$ and $ub_u$ must receive the colors $x-1$ and $x+1$ in an interval coloring of $G$. Moreover, one of $uv$ and $uw$ must also receive color $x$. Without loss of generality, assume that $uv$ receives color $x$. This means that the edge $uw$ must receive either color $x-2$ or $x+2$.}

    \emph{Now, the edges $va_v$ and $vb_v$ must receive colors different from $x$, that differ from each other by exactly 2 (in order to properly color $a_vb_v$) and such that the addition of exactly one color will form an interval with $x$. The available options are $\{x-3,x-1\},\{x-1,x+1\}$ and $\{x+1,x+3\}$.} 

    \emph{First suppose that $va_v$ and $vb_v$ are colored using $x-3$ and $x-1$. Then, in order to form an interval at the vertex $v$, the edge $vw$ must receive color $x-2$ and thus $uw$ receives color $x+2$, in order to form a proper coloring. However, this means that at vertex $w$, the colors $x-2$ and $x+2$ appear with only two other incident edges left to form an interval. Thus, an interval coloring is not possible in this case. A similar argument shows that either of the other two color options for $vv_a$ and $vb_v$ cannot lead to an interval coloring of $G$. Therefore, $\mu_{int}(G) \geq 2$.} 

    \emph{A 2-improper interval coloring $\varphi$ of $G$ can be constructed as follows. Let $\varphi(ia_i) = 1$, $\varphi(ib_i) = 3$, and $\varphi(a_ib_i) = 2$ for $i\in\{u,v,w\}$. Finally, let $\varphi(uv) = \varphi(vw) = \varphi(uw) = 2$. Since only the color 2 appears twice at the vertices $u$, $v$ and $w$, this is a 2-improper interval coloring of the graph. Therefore, $\mu_{int}(G)=2$.} 
\end{example}

It remains an open question whether there exists an infinite family of graphs $G\odot H$ with $G\neq K_1$ whose impropriety attains the upper bound of Corollary~\ref{cor:corona_product}.

As a step towards determining the impropriety of the graph $G \odot T$, where $T$ is a tree, we begin by investigating the special case where $T$ is a caterpillar. Similar to paths, stars and spiders, we construct a 2-improper interval coloring of $K_1 \odot C$, where $C$ is a caterpillar and use this to determine an upper bound on $\mu_{int}(G \odot C)$ for any graph $G$.

\begin{theorem}\label{thm:caterpillar}
Let $C$ be a caterpillar. Then $\mu_{int}(G\odot C) \leq \max\{2,\mu_{int}(G)\}$. 
\end{theorem}
\begin{proof}
Let $v_1,v_2,\dots,v_k$ be the vertices of the path obtained by removing the leaves of the caterpillar $C$, with $v_iv_{i+1}\in E(C)$ for $i \in [k-1]$.

We construct a $\max\{2,\mu_{int}(G)\}$-improper interval coloring of $G \odot C$, $\varphi:E(G\odot C) \to \mathbb{N}$, as follows. 
Let $\hat{\varphi}$ be a $\mu_{int}(G)$-improper interval coloring of $G$. For each edge $e\in E(G)$, let $\varphi(e) = \hat{\varphi}(e)$.

Let $C_x$ be the copy of $C$ corresponding to the vertex $x\in V(G)$. Let $m(x)$ be the maximum value of $\varphi(e)$ over all edges $e\in E(G)$ that are incident with $x$. Now we color each $C_x$ as follows. Let $\ell_{1},\ell_{2},\dots,\ell_{c_1}$ be the leaves in $C_x$ adjacent to the vertex $v_1$. Then, let $\varphi(v_1\ell_i) = m(x)+i$ for $1\leq i\leq c_1$. Let $\varphi(v_1v_2) = m(x)+c_1+1$. 
For each vertex $v_i$, $2 \le i \le k$, let $\ell_{c_{i-1}+2},\ell_{c_{i-1}+3},\dots,\ell_{c_i}$ be the leaves in $C_x$ adjacent to $v_i$. Let $\varphi(v_i\ell_j) = m(x)+j$, for $c_{i-1}+2 \leq j \leq c_i$. Finally, for $2\leq i < k$, let $\varphi(v_iv_{i+1}) = m(x)+c_i+1$.

To complete the coloring of $G \odot C$, a color must be assigned to each edge containing one vertex $x \in V(G)$ and one vertex in $C_x$, the corresponding copy of $C$. Let $\varphi(xv_1) = m(x)+1$ and $\varphi(xv_i) = m(x) + c_{i-1}+1$, for $1<i\leq k$. Finally, let $\varphi(xw) = \varphi(v_iw)$ for every leaf vertex in $w\in V(C)$, where $v_i$ is the vertex adjacent to $w$ in $C_x$. Note that there is no edge $v_iw$ for which $\varphi(v_iw) = m(x) + c_{i-1} + 1$, so only the color $m(x)+1$ is assigned to two distinct edges incident with $x$. 

The coloring $\varphi$ defines an 
improper interval coloring. Each vertex in $G$ has at most $\max\{2,\mu_{int}(G)\}$ incident edges that share the same color, two containing vertices in the corresponding copy of $C$ and at most $\mu_{int}(G)$ containing other vertices in $G$. Each vertex in a copy of $C$ has exactly two incident edges that share the same color. Therefore, $\varphi$ is a $\max\{2,\mu_{int}(G)\}$-improper interval coloring of $G \odot C$. 
\end{proof}

So far, we have determined upper bounds on the impropriety of $G \odot H$ only when $H$ is a cycle or belongs to a specific subclass of trees. For an arbitrary graph $H$, we have the following upper bound on the impropriety of $G \odot H$ in terms of the impropriety of the graph $G$ and the order of the graph $H$. 

\begin{theorem}
\label{thm: corona upper bound}
    For any graphs $G$ and $H$, $\mu_{int}(G \odot H) \leq \max\{\mu_{int}(G), |V(H)|\}$. 
\end{theorem}

\begin{proof}
We construct a $\max\{\mu_{int}(G),|V(H)|\}$-improper interval coloring of $G \odot H$, $\varphi:E(G\odot H) \to \mathbb{N}$, as follows. 
Let $\hat{\varphi}$ be a $\mu_{int}(G)$-improper interval coloring of $G$. For every edge $e\in E(G)$, let $\varphi(e) = \hat{\varphi}(e)$. 

For each vertex $v\in V(G)$, let $H_v$ be the copy of $H$ corresponding to $v$ and let $m(v)$ be the maximum value of $\varphi$ over all edges in $E(G)$ that are incident with $v$. Then, for each vertex $w\in V(H_v)$, let $\varphi(vw) = m(v)+1$. Finally, for every $u,w\in V(H_v)$, let $\varphi(uw) = m(v) + 2$. 

This is an 
improper interval coloring of $G\odot H$. Each vertex in $V(G)$ has at most $\max\{\mu_{int}(G),|V(H)|\}$ incident edges that are assigned the same color. Each vertex $v$ in each copy of $H$ has at most $\max\{1,d_H(v)\}$ incident edges that are assigned the same color. Since $d_H(v) < |V(H)|$ for every vertex $v\in V(H)$, it follows that the coloring $\varphi$ is a $\max\{\mu_{int}(G),|V(H)|\}$-improper interval coloring. 
\end{proof}

We note that this upper bound on $\mu_{int}(G \odot H)$ is sharp in the case when $G$ is interval colorable and $H \cong K_1$. These are not the only graphs attaining this bound, as the corona product $C_3 \odot P_2$ in Example~\ref{ex:corona} has impropriety $\max\{\mu_{int}(C_3),|V(P_2)|\}=2$.

Note that the known sharp examples for \cref{thm: corona upper bound} satisfy $|V(H)| \le 2$.
If we add the condition that $|V(H)| \ge 3$ to \cref{thm: corona upper bound}, 
then we obtain a better upper bound on $\im{G \odot H}$ as follows.
\begin{theorem}
\label{thm:cor_gen_3}
    For any graphs $G$ and $H$ with $|V(H)| \ge 3$, $$\im{G \odot H} \le \max\{\im{G}, \left\lceil \frac{|V(H)|}{3} \right\rceil, \Delta(H)+1\}.$$
\end{theorem}
\begin{proof}
Let $k = \max\{\im{G}, \left\lceil \frac{|V(H)|}{3} \right\rceil, \Delta(H)+1\}$. We construct a $k$-improper interval coloring of $G \odot H$, $\varphi:E(G \odot H) \rightarrow \mathbb{N}$, as follows.
Let $\hat{\varphi}$ be a $\im{G}$-improper interval coloring of $G$.
For every edge $e \in E(G)$, let $\varphi(e) = \hat{\varphi}(e)$.

For each vertex $v \in V(G)$, let $H_v$ be the copy of $H$ corresponding to $v$ and let $m(v)$ be the maximum value of $\varphi$ over all edges in $E(G)$ that are incident with $v$.
Since $|V(H)| \ge 3$, we can partition $V(H_v)$ into three sets $S_{v,1}, S_{v,2}, S_{v,3}$, where $1 \le |S_{v,i}| \le \left\lceil \frac{|V(H_v)|}{3}\right\rceil$ for each $i \in [3]$.
For each $i \in [3]$ and $w \in V(S_{v,i})$, let $\varphi(vw) = m(v)+i$, 
and for every $u,w \in V(H_v)$, let $\varphi(uw) = m(v)+2$.

This is an improper interval coloring of $G \odot H$. 
Each vertex in $V(G)$ has at most $\max\{\im{G}, \left\lceil \frac{|V(H)|}{3}\right\rceil\}$ incident edges that are assigned the same color.
Each vertex $v$ in each copy of $H$ has at most $1+d_H(v)$ incident edges that are assigned the same color.
Since $d_H(v) \le \Delta(H)$ for every vertex $v \in V(H)$, it follows that the coloring $\varphi$ is a $k$-improper interval coloring.
\end{proof}

Related to Theorems~\ref{thm: corona upper bound} and \ref{thm:cor_gen_3}, we raise a question supported by \cref{cor:corona_product} and \cref{thm:caterpillar}.

\begin{question}
For any graphs $G$ and $H$, is it true that $\im{G \odot H} \le \max\{\im{G}, \im{H}+1\}$?
\end{question}

\section{Conclusion and further directions}

In this paper, we have improved the upper bounds on the interval coloring impropriety of several classes of graphs. However, in many cases there remains a significant difference between the known upper bounds presented here and those that have been conjectured.

In addition to Conjecture \ref{conj:2-trees}, the following is of interest.

\begin{question}
    Does there exist a function $f$ depending on $k$, but not on $G$, such that for every $k$-tree $G$, $\im{G} \leq f(k)$?
\end{question}

A graph is \emph{$k$-degenerate} if every subgraph has a vertex of degree at most $k$. Recall that outerplanar graphs and 2-trees are 2-degenerate, that planar graphs are 5-degenerate, and that every planar graph of girth at least 6 is 2-degenerate (see also \cite{Jumnongnit+2021} for an additional class of 2-degenerate planar graphs). Thus studying the impropriety of $k$-degenerate graphs seems of interest. In particular, we pose the following question. 

\begin{question}
    Let $G$ be  a 2-degenerate graph. Is it true that if $\Delta(G) \geq 4$, then $\im{G} \leq \left \lceil \frac{\Delta(G)}{3} \right \rceil$? Moreover, does there exist a constant $C$ such that $\im{G} \leq C$?
\end{question}

In this article we have continued the study of the impropriety of products of graphs initiated by \cite{casselgren2021improper} where the Cartesian products of graphs were studied. 
Considering the strong products of graphs is a natural choice as well. Note that using the same coloring as for the Cartesian product in \cite[Proposition 4.17]{casselgren2021improper} and coloring the rest of the edges $(u_1,v_1)(u_2,v_2)$ using $\alpha(u_1u_2) + \beta(v_1v_2)$, where $\alpha$ (resp.\ $\beta$) is the $\im{G}$-improper interval coloring of $G$ (resp.\ $\im{H}$-improper interval coloring of $H$), we get intervals 
on the edges incident to each vertex, and the obtained bound is 
\begin{equation}
\label{eq:strong}
    \mu_{int}(G \boxtimes H) \leq \max\{\mu_{int}(G), \mu_{int}(H)\} + \Delta(G)\mu_{int}(H) + \Delta(H) \mu_{int}(G).
\end{equation}

\begin{question}
    Can the bound in \eqref{eq:strong} be improved?
\end{question}

Note that the impropriety is not the only way to attack the difficulty of interval colorability. The interval coloring  thickness has also been introduced and studied (see for example \cite{thickness}), and it would be interesting to see if there is any relation between the interval coloring impropriety and the interval coloring thickness of a graph.

\section*{Acknowledgements}
This work was started at the 2022 Graduate Research Workshop in Combinatorics, which was supported in part by NSF grant 1953985 and a generous award from the Combinatorics Foundation.
MacKenzie Carr was supported by the Natural Sciences and Engineering Research Council of Canada (NSERC) Canadian Graduate Scholarship (No. 456422823). 
Eun-Kyung Cho was supported by Basic Science Research Program through the National Research Foundation of Korea (NRF) funded by the Ministry of Education (No. RS-2023-00244543).
Vesna Iršič was supported in part by the Slovenian Research and Innovation Agency (ARIS) under the grants P1-0297, J1-2452, N1-0285 and Z1-50003, and by the European Union (ERC, KARST, 101071836). We thank the anonymous reviewers for their helpful suggestions and comments.

\bibliographystyle{abbrv}
\bibliography{references}
\end{document}